\documentclass[11pt,reqno,tbtags]{amsart}
\usepackage{amssymb}
\usepackage{natbib}
\bibpunct[, ]{[}{]}{;}{n}{,}{,}

\numberwithin{equation}{section}
\allowdisplaybreaks

\newtheorem{theorem}{Theorem}[section]
\newtheorem{lemma}[theorem]{Lemma}

\theoremstyle{definition}

\theoremstyle{remark}

\newenvironment{romenumerate}{\begin{enumerate}
 }{\end{enumerate}}

\newcounter{oldenumi}
{\setcounter{oldenumi}{\value{enumi}}
\begin{romenumerate} \setcounter{enumi}{\value{oldenumi}}}
{\end{romenumerate}}

\newcounter{thmenumerate}

\newcounter{xenumerate}   

\begingroup
  \divide\count255 by 60
  \multiply\count255 by -60
  \advance\count255 by \time
  \ifnum \count255 < 10 \xdef\klockan{\the\count1.0\the\count255}
  \else\xdef\klockan{\the\count1.\the\count255}\fi
\endgroup

\newcommand{\halmos}{\rule{1ex}{1.4ex}}
\newcommand{\proofbox}{\hspace*{\fill}\mbox{$\halmos$}}

\def\rompar(#1){\textup(#1\textup)}    

\def\xexp(#1){e^{#1}}

\newcommand\R{\mathbb R}

\newcommand\N{\mathbb N}  

\newcommand\Z{\mathbb Z}
\newcommand\I{\mathbb I}

\newcounter{CC}

\newcommand\E{\operatorname{\mathbb E{}}}
\renewcommand\P{\operatorname{\mathbb P{}}}

\newcommand\cG{\mathcal G}

\def\[#1]{[\![#1]\!]}

\begin{document}
\title
{A fixed-point approximation for a routing model in equilibrium}

\date{19 June 2013}

\author{Graham Brightwell}
\address{Department of Mathematics, London School of Economics,
Houghton Street, London WC2A 2AE, United Kingdom}
\email{g.r.brightwell@lse.ac.uk}
\urladdr{http://www.maths.lse.ac.uk/Personal/graham/}

\author{Malwina J. Luczak} \thanks{The research of Malwina Luczak is supported by an EPSRC Leadership Fellowship, grant reference EP/J004022/2.}
\address{School of Mathematical Sciences, Queen Mary, University of London}
\email{m.luczak@qmul.ac.uk}


\begin{abstract}
We use a method of Luczak~\cite{l12} to investigate the equilibrium distribution of a dynamic routing model on a network.
In this model, there are $n$ nodes, each pair joined by a link of capacity~$C$.  For each pair of nodes, calls arrive for this pair of endpoints as a Poisson process with rate~$\lambda$.  A call for endpoints $\{u,v\}$ is routed directly onto the link between the two nodes
if there is spare capacity; otherwise $d$ two-link paths between $u$ and $v$ are considered, and the call is routed along a path with lowest maximum load,
if possible.  The duration of each call is an exponential random variable with unit mean.  In the case $d=1$, it was suggested by Gibbens, Hunt and
Kelly in 1990 that the equilibrium of this process is related to the fixed points of a certain equation.  We show that this is indeed the case, for every 
$d \ge 1$, provided the arrival rate $\lambda$ is either sufficiently small or sufficiently large.  In either regime, we show that the equation has a unique
fixed point, and that, in equilibrium, for each $j$, the proportion of links at each node with load $j$ is strongly concentrated around the $j$th
coordinate of the fixed point.
\end{abstract}

\maketitle

\section{Introduction}\label{S:intro}

We consider a routing model in continuous time, where calls have Poisson arrivals and exponential durations, versions of which were studied earlier
in~\cite{aku,ch,ghk,gm,l12,lm08,lmu03,lu}.  This is to a large extent a companion paper to Luczak~\cite{l12}, using the same methods and some of the results 
of that paper; here, our focus is on the properties of the model in equilibrium.

The setting is as follows. For each $n \in \N$, we have a fully connected {\em communication graph} $K_n$, with node set $V_n=\{1, \ldots, n\}$ and link set
$E_n = \{\{u,v\}: 1 \le u < v \le n\}$.  Each link $\{u,v\} \in E_n$ (which we often denote as $uv$) has capacity $C \in \Z^+$.  Calls arrive as a Poisson process with rate $\lambda {n \choose 2}$, where $\lambda$ is a positive constant, and each arriving call has a pair $\{ u,v\}$ of endpoints, chosen
uniformly from among the $\binom{n}{2}$ possible pairs.  The arriving call is routed {\em directly} along the link $uv$, if that link has fewer than $C$ calls currently using it.  If the link $uv$ is fully loaded, then we pick an ordered list $(w_1, \ldots, w_d)$ of $d$ possible intermediate nodes from
$V_n \setminus \{u,v\}$, uniformly at random with replacement, and the call is routed {\em indirectly} along one of the two-link routes $uw_1v, \ldots, uw_dv$,
chosen to minimise the larger of the current loads on its two links, subject to the capacity constraints.  Ties are broken in favour of the first `best' route
in the ordered list.  If none of the $d$ two-link paths is available, then the call is lost.  Call durations are unit mean exponential random variables,
independent of one another and of the arrivals and choices processes.  This routing strategy is called the BDAR (Balanced Dynamic Alternative Routing)
Algorithm.

This model was introduced by Gibbens, Hunt and Kelly~\cite{ghk} in 1990, in the case $d=1$, motivated by a dynamic routing algorithm used at that time by
British Telecom.  Gibbens, Hunt and Kelly did not analyse the model directly, stating that ``It is difficult to analyse the process, even in equilibrium''.
Instead, they analysed a simplified version where the graph structure is neglected: there are $K$ links, and each arriving call chooses one ``direct link''
and two ``alternative links'', uniformly independently at random from the set of all links.  If the direct link has spare capacity, then the call is routed
along that link; if not, then the call is routed on the two alternative links if both have spare capacity.  If a call is accepted onto the two alternative
links, then the durations of that call on each of the two links are independent exponential random variables with mean~1.  These devices are designed to ensure
that the state of this system at time~$t$ can effectively be captured solely by the proportions $\xi_t(j)$ of links with $j$ calls, for $j =0, \dots, C$.  Gibbens, Hunt and Kelly proved a functional law of large numbers for this simplified system, showing in particular that the vector $\xi_t$ converges to the
solution of a certain differential equation.  They suggested that the BDAR network model should behave similarly to their simplified version.  They also noted that this behaviour is not always benign: for certain values of the parameters, the differential equation has multiple fixed points, and they gave a heuristic
argument indicating that, in some such cases, the time taken to ``tunnel'' from the neighbourhood of one fixed point to the other is exponential in the number
$K$ of links.

Following the work of Gibbens, Hunt and Kelly, a number of authors studied various aspects of the BDAR model, and/or variants of it.  Crametz and Hunt~\cite{ch} and Graham and M\'el\'eard~\cite{gm} showed that, with suitable assumptions on the initial conditions, the behaviour of the BDAR model follows that of the
differential equation over a constant time interval in the case $d=1$.  These results are considerably extended by Luczak~\cite{l12}, who gives quantitative
results, and also treats the case of general $d \ge 1$.  

Several variants on the BDAR model have been considered in the literature.  Gibbens, Hunt and Kelly~\cite{ghk} introduce a version, later called the FDAR model, where again $d$ alternative two-link routes are considered, but now the call is routed along the {\em first} route in the list where both links have spare
capacity, if any.  Luczak and Upfal~\cite{lu}, Luczak, McDiarmid and Upfal~\cite{lmu03}, Anagnostopoluos, Kontyoannis and Upfal~\cite{aku}, and Luczak and McDiarmid~\cite{lm08} consider versions where each link has some of its capacity reserved for indirectly routed calls and (possibly) some for directly 
routed calls: these versions have the advantage that, for $d\ge 2$, the model exhibits the ``power of two choices'' phenomenon, leading to far more even allocations of loads across the network.  An explicit illustration of this concerns the minimum value of the capacity $C = C(n,\lambda,d)$ required so that few
or no calls are lost over a long time interval, which is of order $\log \log n$ if $d \ge 2$, some capacity is reserved for indirectly routed calls, and 
the BDAR rule is used to select which indirect route to use.

Luczak and McDiarmid~\cite{lm08} prove results about the equilibrium distribution, for a variant of the model with reserved capacity, in the case where the capacity $C$ tends to infinity with $n$, proving in particular that, in that model, which exhibits the power of two choices, the proportion of links of 
load~$j$ falls off doubly exponentially with~$j$.  To the best of our knowledge, this is the only previous work on the equilibrium distribution of any routing model in this family.  In our model, for parameter values where the approximating differential equation has a unique fixed point, it is natural to expect that
the proportions $\eta(j)$ are well concentrated, and that their expectations are well-approximated by the corresponding coordinate of this fixed point, but up to now there have been no results along these lines.  Analysing the fixed point for our model would show that the proportions of links with each load fall off
singly exponentially in the regime with $\lambda$ small.  For the regime we study with $\lambda$ large, most links have load $C$ in equilibrium.  

In this paper, we prove such a result for two ranges of parameter values.  We assume throughout that the number $d$ of choices and the capacity $C$ are fixed
positive integers, the arrival rate $\lambda$ is also constant, and $n$ tends to infinity.  We consider the cases where $\lambda \le m_1/d$, and where
$\lambda \ge m_2 C^2 d \log(C^2 d))$, for suitable constants $m_1,m_2$.  We prove that, in either of these two regimes, the corresponding sequence of Markov
chains is rapidly mixing, and that, asymptotically, the vector $\eta$ is well concentrated around the unique fixed point of the differential equation
(even more, for each node $v$, the proportion of links incident to $v$ with each load $j$ is well concentrated around the fixed point).  This establishes a
strong form of the `Erlang fixed point approximation' proposed in~\cite{ghk} in these regimes.  See Kelly's survey~\cite{k91} for more information. 
Some cases where $C$, $d$ and $\lambda$ are slowly growing functions of $n$ can also easily be handled by our methods, as in~\cite{l12}: note that in 
these circumstances there is no single differential equation approximating the process.  

Such `law of large numbers in equilibrium' results have been difficult to prove in settings where there are potentially strong
dependencies between system elements (in this context, links).  Here we are able to prove that these dependencies are negligible; it turns out that, in a
suitable sense, links in certain collections evolve approximately independently of one another.


Our technique was developed by Luczak~\cite{l12}, and uses a general concentration of measure inequality proved in that paper.  The heart of our argument
is the analysis of a natural coupling introduced in~\cite{l12}, which in either of the two regimes we study is contractive.
Using this coupling, we prove that two copies of the Markov chain coalesce quickly.  The same coupling is also used to show
that slowly-changing functions of the process (for instance, the number of links incident with a node $v$ with load exactly $j$ at a given time $t$, for
each node $v$ and each $j \in \{0,1, \ldots, C\}$) are well concentrated at each time $t$, with bounds uniform over $t$. This can then be used to show that
such functions are also well concentrated in equilibrium.  (Note that in~\cite{l12}, no attempt was made to show that the coupling is contractive; instead
it was shown that copies of the process starting close to one another do not diverge too quickly.  This is why the concentration of measure bounds in~\cite{l12}
are not uniform in time, but only work over a bounded time interval or an interval of length only slowly increasing with $n$.)
Thanks to the strong concentration of measure, it is then possible to show that the equilibrium balance equations for functions of interest
factorise approximately, leading to a fixed point approximation.

We believe that both the concentration of measure inequalities and the technique as a whole will find further applications.  The method is suitable for proving
laws of large numbers in a wide variety of settings, and -- as we demonstrate in this paper -- it may sometimes be effective in circumstances where more standard techniques cannot be used.

Our technique would require only very minor amendments to handle the FDAR model, where the first available indirect route is chosen, rather than the `best'.  
The method could also easily be adapted to handle variations with some capacity reserved for indirectly routed calls, although the details would be somewhat
different.  We intend to cover such a variant in future work, allowing also $C$ to tend to infinity with $n$.


\medskip

For each link $uv \in E_n$, let $X_t (uv,0)$ denote the number of calls in progress for endpoints $\{u,v\}$ at time $t$ which are routed directly along the link $uv$.  For each pair $\{u,v\}$ of distinct nodes, and each node
$w \in V_n \setminus \{u,v\}$, let $X_t(uwv)$ denote the number of calls in progress at time $t$ which are routed along the path $uwv$ consisting of links
$\{uw, wv\}$, that are in progress at time~$t$.  We call $X_t = ((X_t(uv,0))_{u\not=v}, (X_t (uwv))_{u\not=v, w\not=u,v})$ the {\em load vector}
at time $t$, and let $S = \{0,1, \ldots, C\}^{\binom{n}{2} (n-1)}$ denote the state space, containing the set of all possible load vectors.
Then $X=(X_t)_{t \geq 0}$ is a continuous-time discrete-space Markov chain.  The chain $X$, its state space, and functions
on the state space depend on the number $n$ of nodes, but we will suppress this dependence in our notation.
Given a load vector $x \in S$ and a link $uv \in E_n$, let $x(uv)$ denote the load of link $uv$:
$$
x(uv) = x(uv,0) + \sum_{w \not \in \{u,v\}} \big(x(uvw)+ x(vuw)\big).
$$

We will also work with a jump chain $\widehat{X}_t$ corresponding to $X_t$. The chain we use is not the standard embedded chain but a slower
moving version that will often not change its state at a given step.  Given that the current state, at time $t \in \Z^+$, is $x \in S$,
the next event is an arrival with probability
$$
p(\lambda, C) = \frac{\lambda}{\lambda+ C},
$$
and a {\it potential} departure with probability $1-p(\lambda, C) = C/(\lambda + C)$.
Given that the event is an arrival, each pair of endpoints $\{u,v\}$ is chosen with probability $1/{n \choose 2}$, then each $d$-tuple of intermediate nodes
is chosen with probability $(n\!-\!2)^{-d}$, and the call is routed directly on $uv$ if $x(uv) < C$, and otherwise along the
first two-link route among the $d$ selected that minimises the maximum load of a link.  Given that the event is a potential departure, the calls
currently in the system are enumerated from~1 up to at most $C {n \choose 2}$, and then a number is chosen uniformly at random from
$\{ 1,\dots, C {n \choose 2}\}$.  If there is a call assigned to this number, it departs; otherwise nothing happens.

We claim that the discrete jump chain $\widehat{X}$ and the continuous-time chain $X$ have the same equilibrium distribution.  To see this, note that the $Q$-matrix $Q$ for the continuous-time chain and the transition matrix $P$ of the discrete jump chain satisfy $P = \frac{1}{(\lambda+C)\binom{n}{2}} Q + I$.
Indeed, given any transition between different states in the continuous chain, with rate $r > 0$, the corresponding entry in $P$ is
$r / (\lambda + C) \binom{n}{2}$, by construction; also, the diagonal entries of $P$ are given by $p_{xx} = 1 - \sum_{y \not= x} p_{xy}$, and the diagonal
entries of $Q$ are given by $q_{xx} = - \sum_{y\not= x} q_{xy}$.  Thus the equilibrium vector $\pi$ of the discrete jump chain, which satisfies
$\pi P = \pi$, also satisfies $\pi Q = 0$, and hence is the equilibrium of the continuous-time chain $X$.  As our key results concern the equilibrium
distribution of the chains, we may and shall work throughout with the discrete jump chain.

The same chain is studied in~\cite{l12}, with no assumptions on $\lambda$, over a bounded time interval.  It is shown there
that, subject to suitable initial conditions, the chain does stay close to the solution to a certain differential equation over such a time interval.

For a vector $\xi = (\xi(k): k=0, \ldots, C)$, let $\xi(\le\! j) = \sum_{k=0}^j \xi(k)$.
Define $F: \R^{C+1} \to \R^{C+1}$ by
\begin{eqnarray}
F_0 (\xi) & = & - \lambda \xi(0)- \lambda g_0(\xi)+ \xi(1), \nonumber \\
F_k (\xi) & = & \lambda \xi(k-1) - \lambda \xi(k) + \lambda g_{k-1} (\xi) - \lambda g_k(\xi) \nonumber \\
&&\mbox{} -  k \xi(k) + (k+1) \xi(k+1), \qquad \qquad 0 < k < C, \nonumber \\
F_C (\xi) & = & \lambda \xi (C-1) + \lambda g_{C-1} (\xi) - C \xi(C), \label{eq.F}
\end{eqnarray}
where the functions $g_j$, $j=0, \ldots, C-1$, are given by
\begin{eqnarray}
g_j (\xi) & = & 2 \xi(C) \xi(j) \xi(\le\! j) \sum_{r=1}^d (1-\xi(\le\! j)^2)^{r-1}
(1-\xi(\le\! j\!-\!1)^2)^{d-r} \label{eq.G} \\
&& \mbox{} + 2 \xi(C) \xi(j) \sum_{i=j+1}^{C-1} \xi(i)
\sum_{r=1}^d(1-\xi(\le\! i)^2 )^{r-1} (1-\xi(\le\! i\!-\!1)^2)^{d-r}. \nonumber
\end{eqnarray}

As shown, effectively, by Gibbens, Hunt and Kelly~\cite{ghk}, these equations describe the evolution of their simplified process.  Later, 
Crametz and Hunt~\cite{ch}, Graham and M\'el\'eard~\cite{gm} and Luczak~\cite{l12} showed that the same equations do describe the evolution of the 
true process, under suitable initial conditions.  If we imagine that
links behave completely independently, each having load $j$ at time $t$ with probability $\xi(j)$, then $\lambda g_j(\xi)$ would be proportional to
the rate of arrivals which are indirectly routed onto a link with current load $j$.  The first summand in
(\ref{eq.G}) corresponds to the event that the other link on the indirect route considered has load at most $j$, and the index $r$ in the
summation is the position of this route in the list of $d$ indirect routes considered: each route ahead of it in the list has to have a link of load greater
than $j$, and each route behind it in the list has to have a link of load at least $j$.  The second summand corresponds to the event that the other link on
the route considered has load $i > j$.  In the idealisation, $F_k(\xi)$ is proportional to the expected rate of change in the number of links of
load $k$.  In the case $0<k<C$, the six terms in the expression for $F_k(\xi)$ correspond respectively to: arrivals routed directly onto a link of load $k-1$,
arrivals routed directly onto a link of load $k$, arrivals routed indirectly onto a link of load $k-1$, arrivals routed indirectly onto a link of load $k$,
departures from a link of load $k$, and departures from a link of load~$k+1$.  In particular, in equilibrium, the expected proportions $\eta(j)$ of links
with each load $j$ should be close to a solution of the equation $F(\eta)=0$.  By symmetry, the same approximation should hold if we instead count
the expected proportion of links incident to a given node with load $j$.

Set $\Delta^{C+1} = \{ \xi \in [0,1]^{C+1} : \sum_{j=0}^C \xi(j) = 1\}$.
It is proved in~\cite{l12} that, for any $\xi_0 \in \Delta^{C+1}$, the equation
\begin{equation}
\frac{d\xi_t}{dt} = F(\xi_t)
\label{eq.diff-eq}
\end{equation}
has a unique solution starting from $\xi_0$, with $\xi_t \in \Delta^{C+1}$ for all $t \ge 0$.

Given a load vector $x$, node $v$ and $k \in \Z^+$, let $f_{v,k} (x) = \sum_{w\not=v} \I_{vw}^k(x)$ be the number of links in $x$ with one end $v$
carrying exactly $k$ calls.  The main theorem of~\cite{l12} states that, for each node $v$, and each $k \in \{0, \dots, C\}$, over any bounded 
time interval $[0,t_0]$, the function $\frac{1}{n-1} f_{v,k}(X_t)$ closely follows the solution to the differential
equation~(\ref{eq.diff-eq}), with high probability, provided $n$ is sufficiently large.  The result is quantitative, with
explicit bounds on the fluctuations, in terms of the parameters of the model and also of a function $\phi$ measuring
how ``uniform'' the initial state of the system is.  

%

Our results concern the equilibrium distribution of the chain $X$, or equivalently of the discrete jump chain $\widehat{X}$.
We set $\lambda_0 = \lambda_0(d) = 1/(8d+4)$, and $\lambda_1 = \lambda_1(d,C) = 8000C^2d \log(240C^2d)$.  Let $\pi$ denote the equilibrium
distribution of the chain, and $\mu(x_0,t)$ denote the distribution of $\widehat{X}_t$ conditional on $\widehat{X}_0 = x_0$, for any $x_0 \in S$ and $t \ge 0$.
Let $\widehat Z$ denote a copy of the discrete jump chain in equilibrium.  Our results are as follows.

\begin{theorem} \label{thm.low-high}
For $d, C \in \N$, if either $\lambda < \lambda_0(d)$, or $\lambda \ge \lambda_1(d,C)$, then there is $n_0 = n_0(d,C,\lambda)$
such that the following hold for all $n \ge n_0$.
\begin{enumerate}
\item The discrete jump chain $\widehat{X}$ is rapidly mixing: there are constants $\gamma = \gamma(d,C,\lambda)$ and $K = K(d,C,\lambda)$ such that
$d_{TV}(\mu(x_0,t),\pi) \le K n^2 e^{-\gamma t/n^2}$ for all $x_0 \in S$ and all $t \in [0,n^{5/2}]$.  In particular, $d_{TV}(\mu(x_0,t),\pi) = o(1)$ for
$t \ge 3 \gamma^{-1} n^2 \log n$.
\item There are constants $c_1, c_2 > 0$, depending on $d$, $C$ and $\lambda$, such that, for each node $v$, each $j \in \{0,\dots, C\}$, each~$t$, 
and any $a >0$,
$$
\P_\pi \left( | f_{v,j} (\widehat{Z}_t) - \E_\pi f_{v,j} (\widehat{Z}_t) | > 2a \right) \le 3 \exp\left( - \frac{ a^2 }{ c_1 n + c_2 a} \right).
$$
\item There is a unique solution $\eta^* \in \Delta^{C+1}$ to the equation $F(\eta) = 0$.
\item For all nodes $v$, all $j \in \{ 0, \dots, C\}$, and all $t$,
$$
\Big| \frac{1}{n-1} \E_\pi f_{v,j} (\widehat{Z}_t) - \eta^*(j) \Big| \le 160 d^2 (C+1)^4 \frac{\log n}{\sqrt n}.
$$
\item Let $A$ be the event that $| f_{v,j}(\widehat{Z}_t) - (n-1) \eta^*(j) | \le 200 d^2 (C+1)^4 \sqrt n \log n$, for all nodes $v$ and all
$j \in \{ 0, \dots, C\}$.  Then $\P_\pi (\overline{A}) \le 3Cn^2 e^{-\delta \log^2 n}$ for some constant $\delta = \delta(d,C,\lambda)$.
\end{enumerate}
\end{theorem}

Note that Theorem~\ref{thm.low-high}(5) follows immediately from parts~(2) and~(4), with $a = 30 d^2 (C+1)^4 \frac{\log n}{\sqrt n}$.  As remarked earlier,
parts~(2), (4) and~(5) apply equally to the continuous time chain $Z$ in equilibrium, as the equilibrium distributions of the two chains are the same.  It is
also easy to deduce that the continuous time chain $X$ is rapidly mixing, in time $O(\log n)$, in either of the two regimes considered.

We can deduce immediately from Theorem~\ref{thm.low-high}(5) that, in either regime, the total number of links in the network of each load~$j$ is within
$O(n^{3/2}\log n)$ of $\binom{n}{2} \eta^*(j)$.  We expect that this error term is not optimal.

Theorem~\ref{thm.low-high}(3) is not true without some assumptions on the parameters: Gibbens, Hunt and Kelly~\cite{ghk} found
that, for $d=1$ and $C$ sufficiently large, there is a range of $\lambda$ (apparently roughly of order $C$) where there are multiple solutions
to the fixed-point equation $F(\eta) = 0$.  In such circumstances, we should not expect rapid mixing, or strong concentration of measure in
equilibrium.  However, it should be possible to improve the functions $\lambda_0$ and $\lambda_1$; the results above are likely
to hold for any values of the parameters where the equation $F(\eta) = 0$ does have a unique solution.  For instance, in the special case
$d=1$, $C=1$, the equation has a unique solution for every $\lambda >0$, and we feel that the results above should hold for the whole
range of $\lambda$ in this case.  Preliminary calculations indicate that we can improve the bounds on $\lambda$ at the expense of increased
complexity in the proofs: we intend to return to this in a subsequent paper, but for now our main purposes are to show that we do indeed have strong
concentration of measure in some ranges of the parameters, and to illustrate the power of our methods.

The structure of the paper is as follows.  We begin with some preliminary results and definitions in Section~\ref{S:prelim}, including the definition
of the key coupling we use.  The next two sections are devoted to the analysis of the coupling in the two regimes, $\lambda < \lambda_0(d)$ and
$\lambda \ge \lambda_1(d,C)$ respectively.  In these sections, we establish rapid mixing of the discrete jump chain, and concentration of measure for the
process, both started from a fixed state and in equilibrium.  In particular, we prove parts~(1) and~(2) of Theorem~\ref{thm.low-high} in these sections.
In Section~\ref{S:five}, we use concentration of measure to show that the vector $\zeta$ given by $\zeta(j) = \frac{1}{n-1} \E_\pi f_{v,j}(\widehat{Z}_t)$
satisfies an approximate version of the equation $F(\eta) = 0$.  In the final section, we use a contraction argument to prove parts~(3) and~(4) of
Theorem~\ref{thm.low-high}; we show that the equation $F(\eta) = 0$ has a unique solution $\eta^*$, in either of our two regimes, and that any solution to the approximate version of the equation lies close to $\eta^*$.

\section{Preliminaries}

\label{S:prelim}

Let $\widehat{X}=(\widehat{X}_t)_{t \in \Z^+}$ be a discrete-time Markov chain with a discrete state space $S$ and transition probabilities $P(x,y)$ for
$x,y \in S$. 
We allow $\widehat{X}$ to be lazy, that is we allow $P(x,x) > 0$ for some $x \in S$.
We assume that, for each $x \in S$, the set $N(x) = \{y: P(x,y) > 0\}$ is finite.
Let $\P_{x_0}$ denote the probability measure associated with the chain conditioned on $\widehat{X}_0 = x_0$, and let $\E_{x_0}$ denote the corresponding
expectation operator; then $\E_{x_0} [f(\widehat{X}_t)]$ is the expectation of the function $f$ with respect to measure $\delta_{x_0} P^t$.

The following concentration of measure result is from~\cite{l12}.  In fact, the version in~\cite{l12} is more general, catering for the case when
the inequalities only hold for states in some ``good set'' $S_0 \subset S$.

\begin{theorem}
\label{thm.concb-general}
Let $P$ be the transition matrix of a discrete-time Markov chain with discrete state space $S$, and let $f: S \to \R$ be a function.
Suppose that, for each $x \in S$ and each $i \in \Z^+$, the function $\alpha_{x,i}: S \to \R$ satisfies:
\begin{equation}
\label{cond-gen-3}
|\E_x [f(\widehat{X}_i)] -  \E_y [f(\widehat{X}_i)]| \le \alpha_{x,i}(y),
\end{equation}
for all $y \in S$.
Suppose also that the sequence $(\alpha_i: i \in \Z^+)$ of positive constants satisfies:
\begin{equation}
\label{cond-gen-4}
\sup_{x \in S} (P \alpha_{x,i}^2)(x) \le \alpha_i^2.
\end{equation}
Let $t >0$, and set $\beta = 2\sum_{i=0}^{t-1} \alpha_i^2$.  Suppose also that $\widehat{\alpha}$ is such that
\begin{equation}
\label{cond-gen-5}
\sup_{0 \le i \le t -1}  \sup_{x \in S, y \in N(x)} \alpha_{x,i} (y)\le \widehat{\alpha}.
\end{equation}
Then, for all $a >0$,
\begin{equation}
\label{ineq.concc}
\P_{x_0}\Big ( |f(\widehat{X}_t)-\E_{x_0} [f(\widehat{X}_t)] |\ge a \Big )\le 2e^{-a^2/(2\beta+\frac43\widehat{\alpha}a)}.
\end{equation}
\end{theorem}


\medskip

Given two load vectors $x,y$, the $\ell_1$-distance between them is
$$
\|x-y\|_1 = \sum_{uv \in E_n} |x(uv,0) - y(uv,0)| + \sum_{uv \in E_n, w \not= u,v} |x(uwv) - y(uwv)|,
$$
the sum of the differences between $x$ and $y$ in loads of all possible routes.  Then $\|\cdot\|_1$ is a metric on $S$.
For $v \in V_n$, we will also consider the function
\begin{eqnarray*}
\|x-y\|_v & = &
\sum_{u\not = v}|x(uv,0) -y(uv,0)| \\
&&\mbox{} + \sum_{u,w}  |x (uwv) - y(uwv) | + \sum_{\{u,w\}} |x(uvw)-y(uvw)|,
\end{eqnarray*}
where the sums are over distinct nodes $u,w \not= v$.


Consider the following family of Markovian couplings $(\widehat{X},\widehat{Y})$ of pairs of copies $\widehat{X},\widehat{Y}$ of
the discrete jump chain starting from states $x_0,y_0$ respectively, where $x_0,y_0 \in S$.
Let $t \ge 0$, and let $x,y$ be states in $S$.  Given that $\widehat{X}_{t-1} = x$ and $\widehat{Y}_{t-1} = y$, the transition at time $t$ (from state
$(\widehat{X}_{t-1}, \widehat{Y}_{t-1})$ to $(\widehat{X}_t, \widehat{Y}_t)$) is an arrival in both $\widehat{X}$ and $\widehat{Y}$, or a potential
departure in both $\widehat{X}$ and $\widehat{Y}$.  Given that the transition is an arrival, we choose the same call endpoints and the same $d$-tuple of
intermediate nodes in both.  Also, given that the transition is a potential departure, we pair calls occupying the same route in both $\widehat{X}$
and $\widehat{Y}$, and call such pairs {\em matched}, as much as possible.  We also pair off the remaining calls arbitrarily, as much as possible, in some
fashion depending only on $x$ and $y$.  (We can pair off all the calls if $\|x\|_1 = \|y\|_1$; otherwise some calls remain unpaired in the process
that has more calls.)  Then we always choose paired calls for a simultaneous departure; any unpaired calls depart alone.  This is
achieved by assigning to each pair, and also to each unpaired call, a distinct number in $\{1, \dots, C{n \choose 2}\}$.  Then, if the transition at time $t$ is a potential departure, we choose a uniformly random number from $\{ 1, \dots, C{n \choose 2}\}$.  If the number corresponds to a pair of calls, one in
$\widehat X$ and one in $\widehat{Y}$, both depart; if it corresponds to an unpaired call, this call departs and there is no change in the other state;
otherwise, nothing happens.
The process $\widehat{W} = (\widehat{W}_t)$ given by $\widehat{W}_t = (\widehat{X}_t, \widehat{Y}_t)$ is a Markov chain adapted to its natural filtration,
${\mathcal G}_t = \sigma (\widehat{X}_s,\widehat{Y}_s: s \le t)$.

\section{Analysis of coupling -- low arrival rate}

\label{S:coupling-anal}

In this section, we prove that, provided $\lambda$ is sufficiently small, for two coupled copies $\widehat{X}$ and $\widehat{Y}$ of the discrete jump chain,
the distance $\|\widehat{X}_t-\widehat{Y}_t\|_1$ is contractive.  We also prove a similar result for each $\|\widehat{X}_t-\widehat{Y}_t\|_v$.  The aim is
firstly to show that the chain is rapidly mixing, and also then to show that quantities of interest are well concentrated, both starting in a fixed state
uniformly over long time intervals, and in equilibrium.

Our main use of the lemma below will be in the case $\lambda < \lambda_0(d) = 1/(8d+4)$, when it indeed states that the distance is contractive under the
coupling.  We shall also use this lemma in the next section, to show that the distance is only mildly expansive even for much larger $\lambda$: in this context, the result is very similar to a lemma of Luczak~\cite{l12}.

\begin{lemma} \label{lem.coupling}
Suppose that $n \ge \max(d^2,6)$. 
Let $\widehat{X}$ and $\widehat{Y}$ be two copies of the discrete jump chain, coupled as in the previous section.  Then we have:
$$
\E (\|\widehat{X}_t-\widehat{Y}_t\|_1 \mid \cG_0) \le \left( 1 - \frac{1-(8d+4)\lambda}{(\lambda+C)\binom{n}{2}} \right)^t \|\widehat{X}_0 - \widehat{Y}_0\|_1.
$$
Also, for each node $v$,
\begin{eqnarray*}
\lefteqn{\E ( \|\widehat{X}_t-\widehat{Y}_t\|_v \mid \cG_0 )} \\
&\le& \left( 1 - \frac{1 - (8d+4)\lambda}{(\lambda+C)\binom{n}{2}} \right)^t \left( \| \widehat{X}_0 - \widehat{Y}_0 \|_v +
\frac{50d^2\lambda t}{(\lambda+C)n^3}\| \widehat{X}_0 - \widehat{Y}_0 \|_1 \right).
\end{eqnarray*}
\end{lemma}

\begin{proof}
We first give bounds on the expected change in distance on one step of the discrete jump chain.  So we shall bound
$\E (\| \widehat{X}_{t+1} - \widehat{Y}_{t+1} \|_1 \mid \cG_t)$ and $\E (\| \widehat{X}_{t+1} - \widehat{Y}_{t+1} \|_v \mid \cG_t)$,
in terms of $\| \widehat{X}_t - \widehat{Y}_t \|_1$ and $\| \widehat{X}_t - \widehat{Y}_t \|_v$.

To begin with, suppose that $\widehat{X}_t$ and $\widehat{Y}_t$ differ in one call.  Consider first the case where $\widehat{X}_t$ and $\widehat{Y}_t$ are
identical except that $\widehat{X}_t$ contains a call on a direct link $uv$ that is not present in $\widehat{Y}_t$, so
$\widehat{X}_t(uv,0) = \widehat{Y}_t(uv,0)+1$.

We consider the expected value of $\|\widehat{X}_{t+1}-\widehat{Y}_{t+1}\|_1$, conditioned on $\cG_t$.  The distance decreases to~0 on the departure of the
single unpaired call, which occurs with probability $\frac{1}{(\lambda+C)\binom{n}{2}}$.  On any other potential departure, the distance remains equal to~1.
On an arrival, the distance remains at most~1 if the pair $\{u,v\}$ is chosen as endpoints for the arriving call: either $\widehat{X}_t(uv) < C$, in which
case the call is routed directly in both $\widehat{X}$ and $\widehat{Y}$, or $\widehat{X}_t(uv) = C = \widehat{Y}_t(uv) + 1$, in which case
the call is routed directly in $\widehat{Y}$, and the only difference between $\widehat{X}_{t+1}$ and $\widehat{Y}_{t+1}$ is constituted by a new call routed
indirectly in $\widehat{X}$ (or coalescence occurs in the event that the call is not routed at all in $\widehat{Y}$).  If endpoints $\{w,z\}$ are chosen for the
arriving call, where $\{u,v\} \cap \{w,z\} = \emptyset$, then the new call is routed the same way in both $\widehat{X}$ and $\widehat{Y}$, and the distance
remains~1.  So the only way that the distance can increase is if an arriving call is for endpoints $\{w,z\}$, where $|\{u,v\} \cap \{w,z\}| = 1$, say $w=u$.
In this case, if $\widehat{X}_t(wz) = \widehat{Y}_t(wz) = C$, and $v$ is among the $d$ intermediate nodes selected for a possible indirect route, then it is
possible for different indirect routes to be chosen -- specifically, the route via $v$ may be chosen in $\widehat{Y}$, and some other route may be preferred
in $\widehat{X}$ -- and in this case $\|\widehat{X}_{t+1}-\widehat{Y}_{t+1}\|_1 = 3$.  The probability that the next transition is an arrival, one of $u$ and
$v$ is among the selected endpoints $\{w,z\}$, and the other is among the $d$ intermediate nodes, is at most
$$
\frac{\lambda}{\lambda+C} \frac{2(n-2)}{\binom{n}{2}} \frac{d}{n-2} = \frac{\lambda}{\lambda+C} \frac{2d}{\binom{n}{2}}.
$$
Therefore
$$
\E (\|\widehat{X}_{t+1}-\widehat{Y}_{t+1}\|_1 \mid \cG_t) \le 1 - \frac{1}{(\lambda+C)\binom{n}{2}} + 2\frac{\lambda}{\lambda+C} \frac{2d}{\binom{n}{2}}
= 1 - \frac{1-4d\lambda}{(\lambda+C)\binom{n}{2}}.
$$

We now perform a similar analysis in the case where $\widehat{X}_t$ and $\widehat{Y}_t$ differ by the presence of an extra indirectly routed call in
$\widehat{X}_t$, say $\widehat{X}_t(uwv) = \widehat{Y}_t(uwv)+1$.  As before, with probability $\frac{1}{(\lambda+C)\binom{n}{2}}$,
this extra call departs, and the distance drops to~0, while on any other potential departure the distance between $\widehat{X}$ and $\widehat{Y}$ remains equal to~1.
The distance can increase on an arrival, by at most~2, only if one of the links
$uw$, $vw$ is among the $2d+1$ links whose capacity is (potentially) inspected in deciding how to route the arriving call.  The probability of
this event is at most $\frac{\lambda}{\lambda+C} \frac{4d+2}{\binom{n}{2}}$.  Thus
$$
\E (\|\widehat{X}_{t+1}-\widehat{Y}_{t+1}\|_1 \mid \cG_t) \le 1 - \frac{1}{(\lambda+C)\binom{n}{2}} + 2\frac{\lambda}{\lambda+C} \frac{4d+2}{\binom{n}{2}}
= 1 - \frac{1-(8d+4)\lambda}{(\lambda+C)\binom{n}{2}}.
$$
Overall, we have that, for coupled copies of the chain at distance~1 at time~$t$,
$$
\E (\|\widehat{X}_{t+1}-\widehat{Y}_{t+1}\|_1 \mid \cG_t) \le 1 - \frac{1-(8d+4)\lambda}{(\lambda+C)\binom{n}{2}}.
$$
Hence, for coupled copies $\widehat{X}$ and $\widehat{Y}$ at any distance,
$$
\E (\|\widehat{X}_{t+1}-\widehat{Y}_{t+1}\|_1 \mid \cG_t) \le \left( 1 - \frac{1-(8d+4)\lambda}{(\lambda+C)\binom{n}{2}} \right) \|\widehat{X}_t - \widehat{Y}_t\|_1,
$$
and so
\begin{equation} \label{eq.shrink}
\E (\|\widehat{X}_t-\widehat{Y}_t\|_1 \mid \cG_0) \le \left( 1 - \frac{1-(8d+4)\lambda}{(\lambda+C)\binom{n}{2}} \right)^t \|\widehat{X}_0 - \widehat{Y}_0\|_1.
\end{equation}

Now we fix a node $v$, and consider $\E (\|\widehat{X}_{t+1}-\widehat{Y}_{t+1}\|_v \mid \cG_t)$.
We start with the case where 
the only difference between
$\widehat{X}_t$ and $\widehat{Y}_t$ consists of a single call in $\widehat{X}$, directly or indirectly routed, not involving $v$.  In this case, in
order for $\|\widehat{X}_{t+1}-\widehat{Y}_{t+1}\|_v$ to be non-zero, it must be the case that the next transition is an arrival, in which the loads on
some $2d+1$ links are inspected, including one carrying the extra call in $\widehat{X}$, and that $v$ is among the other at most $d$ nodes selected.  The
probability of this event is at most
$\frac{\lambda}{\lambda+C} \frac{2(2d+1)}{\binom{n}{2}} \frac{d}{n-2}$.
Again, the maximum value of $\| \widehat{X}_{t+1} - \widehat{Y}_{t+1} \|_v$ is~2, and so, in this case,
\begin{eqnarray*}
\E ( \|\widehat{X}_{t+1}-\widehat{Y}_{t+1}\|_v \mid \cG_t ) &\le& 2 \frac{\lambda}{\lambda+C} \frac{2(2d+1)}{\binom{n}{2}} \frac{d}{n-2} \\
& \le & \frac{\lambda}{\lambda+C} \frac{12d^2}{\binom{n}{2}(n-2)}.
\end{eqnarray*}

Now suppose that $\|\widehat{X}_t - \widehat{Y}_t\|_v = \|\widehat{X}_t - \widehat{Y}_t\|_1 =1$, and that the only difference between $\widehat{X}_t$ and
$\widehat{Y}_t$ is the presence of a single call in $\widehat{X}$, directly or indirectly routed, that does involve $v$.  As usual, the departure of this
extra call, which occurs with probability $\frac{1}{(\lambda+C)\binom{n}{2}}$, reduces the distance to~0, while no other departure changes
$\|\widehat{X}-\widehat{Y}\|_v$.

To analyse the arrivals, we distinguish two cases.  The first is where $v$ is an endpoint of the extra call, which utilises a link $vw$ and possibly
a further link $wu$.  In this case, $\| \widehat{X} - \widehat{Y} \|_v$ can increase, by at most~2, only if the link $vw$ is one of the $2d+1$
links inspected for the arriving call, an event with probability at most $\frac{\lambda}{\lambda+C} \frac{2d+1}{\binom{n}{2}}$, or if the
arriving call has $u$ as one endpoint, and $v$ and $w$ as two of the potential intermediate nodes, an event with probability at most
$\frac{\lambda}{\lambda+C} \frac{2}{n} \frac{d(d-1)}{(n-1)^2}$.  Thus, in this case, for $n \ge d^2$,
$$
\E ( \|\widehat{X}_{t+1}-\widehat{Y}_{t+1}\|_v \mid \cG_t ) \le 1 - \frac{1}{(\lambda+C)\binom{n}{2}} + 2 \frac{\lambda}{\lambda+C} \frac{2d+2}{\binom{n}{2}}
= 1 - \frac{1-(4d+4)\lambda}{(\lambda+ C)\binom{n}{2}}.
$$
The other case is where $v$ is the intermediate node of the extra call, with endpoints $u$ and $w$.  Here, the distance can increase, by at most~2,
only if one of the two links $vu$, $vw$ is among the $2d+1$ links inspected for the arriving call, an event with probability at most
$\frac{\lambda}{\lambda+C} \frac{4d+2}{\binom{n}{2}}$.  Thus we have
$$
\E ( \|\widehat{X}_{t+1}-\widehat{Y}_{t+1}\|_v \mid \cG_t ) \le 1 - \frac{1}{(\lambda+C)\binom{n}{2}} + 2 \frac{\lambda}{\lambda+C} \frac{4d+2}{\binom{n}{2}}
= 1 - \frac{1-(8d+4)\lambda}{(\lambda+ C)\binom{n}{2}}.
$$

Therefore we obtain, whenever $\widehat{X}$ and $\widehat{Y}$ are coupled,
\begin{eqnarray*}
\E ( \|\widehat{X}_{t+1}-\widehat{Y}_{t+1}\|_v \mid \cG_t ) &\le& \left( 1 - \frac{1-(8d+4)\lambda}{(\lambda+ C)\binom{n}{2}} \right)
\| \widehat{X}_t - \widehat{Y}_t \|_v \\
&&\mbox{} + \frac{\lambda}{\lambda+C} \frac{12d^2}{\binom{n}{2}(n-2)} \| \widehat{X}_t - \widehat{Y}_t \|_1.
\end{eqnarray*}

Applying our earlier upper bound (\ref{eq.shrink}) on $\E \| \widehat{X}_t - \widehat{Y}_t \|_1$, we obtain:
\begin{eqnarray*}
\lefteqn{\E ( \|\widehat{X}_{t+1}-\widehat{Y}_{t+1}\|_v \mid \cG_t ) \le \left( 1 - \frac{1-(8d+4)\lambda}{(\lambda+ C)\binom{n}{2}} \right) 
\E\| \widehat{X}_t - \widehat{Y}_t \|_v }\\
&&\mbox{} + \frac{\lambda}{\lambda+C} \frac{12d^2}{\binom{n}{2}(n-2)} \left( 1 - \frac{1-(8d+4)\lambda}{(\lambda+C)\binom{n}{2}} \right)^t \|\widehat{X}_0 - \widehat{Y}_0\|_1
\end{eqnarray*}
and thus, by induction,
\begin{eqnarray*}
\lefteqn{\E ( \|\widehat{X}_t-\widehat{Y}_t\|_v \mid \cG_0 )} \\
&\le& \left( 1 - \frac{1-(8d+4)\lambda}{(\lambda+ C)\binom{n}{2}} \right)^t \| \widehat{X}_0 - \widehat{Y}_0 \|_v \\
&&\mbox{} +\frac{\lambda}{\lambda+C} \frac{12d^2t}{\binom{n}{2}(n-2)} \|\widehat{X}_0 - \widehat{Y}_0\|_1
\left( 1 - \frac{1-(8d+4)\lambda}{(\lambda+C)\binom{n}{2}} \right)^{t-1} \\
&\le& \left( 1 - \frac{1-(8d+4)\lambda}{(\lambda+C)\binom{n}{2}} \right)^t \left( \| \widehat{X}_0 - \widehat{Y}_0 \|_v
+ \frac{50d^2\lambda t}{(\lambda +C)n^3}\| \widehat{X}_0 - \widehat{Y}_0 \|_1 \right),
\end{eqnarray*}
for $n \ge \max(d^2,6)$, as desired. 
\end{proof}



The next step is to apply Theorem~\ref{thm.concb-general} to our Markov chain.
For nodes $u$ and $v$, we say that a function $f:S \to \R$ is {\em $(u,v)$-Lipschitz} if 
$|f(x) - f(y)| \le \| x-y \|_u + \| x-y \|_v$ for all states $x$ and $y$.

\begin{lemma} \label{lem.concentration}
Set $\rho = 1 - \lambda/\lambda_0(d)$, and assume that $\lambda < \lambda_0(d)$, so that $\rho > 0$.
For any nodes $u$ and $v$, let $f$ be a $(u,v)$-Lipschitz function.   
Then, for all sufficiently large $n$, all $a > 0$, all $t\ge 0$, and all initial states $x_0$,
$$
\P_{x_0}\Big ( |f(\widehat{X}_t)-\E_{x_0} [f(\widehat{X}_t)] |\ge a \Big ) \le 2\exp \left( - \frac{a^2}{40dCn/\rho + 3a}\right).
$$
\end{lemma}

\begin{proof}
Our aim is to apply Theorem~\ref{thm.concb-general}, so we need to bound the various quantities appearing in that theorem, for our function $f$.
The second part of Lemma~\ref{lem.coupling}, together with the assumption on $f$, allows us to set
$$
\alpha_{x,i}(y) = \left( 1 - \frac{\rho}{(\lambda+C)\binom{n}{2}} \right)^i \left( \| x-y \|_u + \| x-y \|_v +
\frac{100d^2\lambda i}{(\lambda+C)n^3}\| x - y \|_1 \right),
$$
provided $n$ is sufficiently large.

Let $P$ be the transition matrix of our Markov chain.  If $P(x,y) >0$, then $\|x-y \|_u, \| x-y \|_v \le \| x - y \|_1 \le 1$, and moreover, for any
state~$x$,
$$
\sum_{y: \| x - y \|_v = 1} P(x,y) \le \frac{Cn}{(\lambda+C)\binom{n}{2}} + \frac{\lambda (d+2)}{n(\lambda+C)} \le \frac{d+2}{n},
$$
and the same is true with $v$ replaced by $u$.
Therefore
\begin{eqnarray*}
(P \alpha_{x,i}^2)(x) &=& \sum_y P(x,y) \alpha_{x,i}(y)^2 \\
&\le& 3 \sum_y P(x,y) \left( 1 - \frac{\rho}{(\lambda+C)\binom{n}{2}} \right)^{2i} \\
&&\mbox{} \times \left( \| x-y \|_u^2 + \| x - y \|_v^2  + \frac{10000d^4\lambda^2i^2}{C^2n^6}\| x - y \|_1^2 \right) \\
&\le& 3 \left( 1 - \frac{\rho}{(\lambda+C)\binom{n}{2}} \right)^i \left( 2 \frac{d+2}{n} + \frac{10000d^4\lambda^2i^2}{(\lambda+C)^2n^6} \right).
\end{eqnarray*}
Set $\displaystyle q = 1 - \frac{\rho}{(\lambda+C)\binom{n}{2}}$.
Thus, in Theorem~\ref{thm.concb-general}, we may take
$$
\alpha_i^2 = 6 q^i \left( \frac{d+2}{n} + \frac{5000d^4\lambda^2i^2}{C^2n^6} \right)
$$
and
\begin{eqnarray*}
\beta = \sum_{i=0}^{t-1} \alpha_i^2 &=& 6 \frac{d+2}{n} \sum_{i=0}^{t-1} q^i + \frac{30000d^4\lambda^2}{C^2n^6} \sum_{i=0}^{t-1} i^2 q^i \\
&\le& \frac{6d+12}{n( 1 - q)} + \frac{30000d^4\lambda^2}{C^2n^6} \frac{q(q+1)}{(1-q)^3} \\
&=& \frac{6d+12}{n} \frac{(\lambda+C) \binom{n}{2}}{\rho} + \frac{30000d^4\lambda^2}{C^2n^6} q(q+1) \left( \frac{(\lambda+C) \binom{n}{2}}{\rho} \right)^3 \\
&\le& \frac{(3d+6)(\lambda+C)n}{\rho} + \frac {7500d^4\lambda^2 (\lambda+C)^3}{C^2\rho^3}.
\end{eqnarray*}
For sufficiently large $n$, we have $\beta \le 20dCn/\rho$, uniformly for $t \ge 0$.
Also
$$
\sup_{0\le i \le t-1} \sup_{x\in S, y\in N(x)} \alpha_{x,i}(y) =
\max_{i\ge 0} \left( 1 - \frac{\rho}{(\lambda+C) \binom{n}{2}} \right)^i \left( 2 + \frac{100d^2\lambda i}{Cn^3} \right) =2,
$$
for sufficiently large $n$, so we may take $\widehat{\alpha} = 2$, uniformly for $t \ge 0$.

Hence, by Theorem~\ref{thm.concb-general}, we have, for all sufficiently large $n$, all $a > 0$, all $t\ge 0$, and all initial states $x_0$,
\begin{eqnarray*}
\P_{x_0}\Big ( |f(\widehat{X}_t)-\E_{x_0} [f(\widehat{X}_t)] |\ge a \Big ) &\le&
2e^{-a^2/(2\beta+\frac43\widehat{\alpha}a)} \\
&\le& 2\exp \left( - \frac{a^2}{40dCn/\rho + 3a}\right),
\end{eqnarray*}
as claimed.
\end{proof}

Now let $\widehat{X}$ and $\widehat{Z}$ be two copies of the discrete jump chain, with $\widehat{X}$ started in a fixed state $x_0$, and $\widehat{Z}$
started in equilibrium, coupled as above.  Then we have, for $\lambda < \lambda_0(d)$ and $n$ sufficiently large,
\begin{eqnarray*}
\E \| \widehat{X}_t - \widehat{Z}_t \|_1 &\le& \left( 1 - \frac{1- \lambda/\lambda_0(d)}{(\lambda+C)\binom{n}{2}} \right)^t
\E \|\widehat{X}_0 - \widehat{Z}_0\|_1 \\
&\le& \exp\left( - \frac{2(1- \lambda/\lambda_0(d))}{\lambda+C} \frac{t}{n^2} \right) C \binom{n}{2}.
\end{eqnarray*}
Hence there are constants $K$ and $\gamma$ such that
$\E \| \widehat{X}_t - \widehat{Z}_t \|_1 \le K n^2 e^{-\gamma t / n^2}$.
Recall that $\mu(x_0,t)$ is the distribution of the chain, started in state $x_0$, after $t$ steps, and that
$\pi$ is the equilibrium distribution.  We have, for any $t$,
$$
d_{TV}(\mu(x_0,t),\pi) \le d_W(\mu(x_0,t),\pi) \le \E \| \widehat{X}_t - \widehat{Z}_t \|_1 \le K n^2 e^{-\gamma t / n^2}.
$$
Here $d_W(\cdot,\cdot)$ denotes the Wasserstein distance between the two distributions, which is equal to the infimum, over all couplings between random
variables $U$ and $V$ having the two distributions, of $\E \| U-V \|_1$.  The first inequality above holds since $\| \widehat{X}_t - \widehat{Y}_t\|_1$ takes
integer values.  Thus we have proved Theorem~\ref{thm.low-high}(1) in the case $\lambda < \lambda_0(d)$.  


\medskip

We have concentration of measure for the process started from a fixed state, as well as rapid mixing of the
process to equilibrium.  It is now possible to deduce concentration of measure in equilibrium, as follows.
Let $\widehat{Z}$ be a copy of the process in equilibrium, and let $\widehat{X}$ be a copy of the process started in any fixed
state $x$.  For nodes $u$ and $v$, let $f$ be a $(u,v)$-Lipschitz function.
Fix $a > 0$, and choose $t >0$ so large that $Kn^2 e^{-\gamma t/n^2}$ is smaller than both $a/2$ and
$\exp\left( -\frac{a^2}{40dCn/\rho+3a}\right)$.  We have
$$
|f(\widehat{Z}_t) - \E f(\widehat{Z}_t)| \le |f(\widehat{Z}_t) - f(\widehat{X}_t)| + |f(\widehat{X}_t) - \E f(\widehat{X}_t)| +
|\E f(\widehat{X}_t) - \E f(\widehat{Z}_t)|,
$$
and also
$$
|\E f(\widehat{X}_t) - \E f(\widehat{Z}_t)| \le \E |f(\widehat{X}_t) - f(\widehat{Z}_t)| \le \E \| \widehat{X}_t - \widehat{Z}_t \|_u +
\E \| \widehat{X}_t-\widehat{Z}_t\|_v 
$$
$$
\le 2Kn^2 e^{-\gamma t/n^2} \le a.
$$
Furthermore, $\P(|f(\widehat{Z}_t) - f(\widehat{X}_t)| > 0) \le \P(\widehat{X}_t \not= \widehat{Z}_t) \le Kn^2 e^{-\gamma t/n^2}$.
Thus
\begin{eqnarray}
\lefteqn{
\P (|f(\widehat{Z}_t) - \E f(\widehat{Z}_t)| > 2a)
} \nonumber\\
&\le& \P(|f(\widehat{Z}_t) - f(\widehat{X}_t)| > 0) + \P(|f(\widehat{X}_t) - \E f(\widehat{X}_t)| > a) \nonumber \\
&\le& Kn^2 e^{-\gamma t/n^2} + 2 \exp\left( - \frac{a^2}{40dCn/\rho + 3a}\right) \nonumber \\
&\le& 3 \exp\left( - \frac{a^2}{40dCn/\rho + 3a}\right). \label{eq:conc-equilibrium}
\end{eqnarray}
Applying (\ref{eq:conc-equilibrium}) to the function $f = f_{v,j}$, for any $j \in \{ 0, \dots, C\}$, proves part~(2) of Theorem~\ref{thm.low-high} in the
case $\lambda < \lambda_0$.

Later, we shall use this result in the case $a = \frac14 \sqrt{n} \log n$, where we have
\begin{equation} \label{eq.conc}
\P (|f(\widehat{Z}_t) - \E f(\widehat{Z}_t)| > \frac12 \sqrt n \log n) \le 3 e^{-\delta \log^2 n},
\end{equation}
and we may take the constant $\delta>0$ to be $\rho/800dC$.

\section{Analysis of coupling -- high arrival rate}

We now turn attention to our other regime, where the arrival rate is very high.  Recall that $\lambda_1(d,C) = 8000C^2d \log(240C^2d)$.
Let $R$ be the set of states $x \in S$ such that, for each node $v$,
$f_{v,C}(x) \ge (n-1) (1- \frac{1}{60Cd})$.  In other words, $R$ is the set of states such that, at each node $v$, there are at most $(n-1)/60Cd$ links that
are not fully loaded.  Set $s = 3 \binom{n}{2} \frac{C(\lambda+C)}{\lambda} \log(240C^2d)$. 
We shall first show that, for $\lambda \ge \lambda_1(d,C)$, for any starting state,
after a ``burn-in'' period of $s$ steps, with high probability, the chain remains in $R$ for a long period of time.

For this purpose, we need tail estimates on the probability of a sum of Boolean random variables that are not necessarily independent.  We use the following
special case of a result of Panconesi and Srinivasan~\cite{ps97}.

\begin{theorem} \label{thm.ps97} Let $W_1, \dots, W_N$ be indicator random variables such that, for some $\delta \in (0,1/2)$, and for all subsets
$D$ of $\{ 1, \dots, N\}$, $\P (W_i = 1 \mbox{ for all } i \in D) \le \delta^{|D|}$.  Then $\P (\sum_{i=1}^N W_i \ge 2\delta N) \le e^{-\delta N /3}$.
\end{theorem}

\begin{lemma} \label{lem.R}
Suppose $\lambda \ge \lambda_1(d,C)$, and let $\kappa >0$ be any fixed constant.
Then $\P\big(\widehat{X}_t \in R \mbox{ for every } t \in [s, n^\kappa] \big) \ge 1- 2Cn^{\kappa +1}e^{-(n-1)/1500C^3d}$.
\end{lemma}

\begin{proof}
Fix a node $v$ and a time $t \ge s$.  Let $r = \lfloor 3 \binom{n}{2} \frac{\lambda+C}{\lambda} \log(240C^2d) \rfloor = \lfloor s/C \rfloor$.
Consider a link $uv$ incident with $v$, and the time intervals $I_C = [t-Cr, t-(C-1)r), \dots, I_2 = [t-2r, t-r), I_1 = [t-r,t)$.
Define the events $E_j(u), F_j(u)$, for $j =1, \dots, C$, as follows.
\begin{itemize}
\item $E_j(u)$: there is at most one arrival during $I_j$ with endpoints $\{u,v\}$;
\item $F_j(u)$: there are at least $j$ departures of calls including the link $uv$ during $I_j \cup \cdots \cup I_1 = [t-jr,t)$.
\end{itemize}
We claim that, if link $uv$ is not fully loaded at time~$t$, then one of $E_1(u)$, \dots, $E_C(u)$, $F_1(u)$, \dots, $F_C(u)$ occurs.
Indeed, suppose none of these events occur, and also that the link was never fully loaded during the interval
$I = I_C \cup \cdots \cup I_1 = [t-Cr,t)$; then every arrival during $I$ with endpoints $\{u,v\}$ was routed directly onto the link $uv$, so
$\overline {E_1(u)}, \dots, \overline{E_C(u)}$ imply that at least $2C$ arriving calls were routed directly onto the link $uv$ during $I$, while
$\overline {F_C(u)}$ states that at most $C-1$ calls occupying link $uv$ departed during $I$, so there are $C+1$ more calls on $uv$ at time $t$ than at
time $t-Cr$, which is a contradiction, since the link has capacity~$C$.  Thus, if none of the $E_j(u)$ or $F_j(u)$ occur, then at some point during $I$ the
link $uv$ was fully loaded.  Consider the last time during $I$ that the link was fully loaded, say at some time $t^*$ during interval $I_j$.  Then all calls
arriving for endpoints $\{u,v\}$ during $I_{j-1} \cup \cdots \cup I_1$ were routed onto $uv$, so there were at least $2(j-1)$ arrivals onto $uv$ during this
interval, and at most $j-1$ departures over the interval $[t^*,t)$, so there are at least as many calls on $uv$ at time $t$ as there were at time $t^*$, and
hence the link is fully loaded at time~$t$.

We now seek upper bounds on the probability that many of the events $E_j(u)$ and $F_j(u)$ occur.
We start with $E_j(u)$, for any $j \in \{1,\dots, C\}$.  The number $N_j(u)$ of arrivals with endpoints $u$ and $v$ over the interval $I_j$ is a binomial random
variable with parameters $\big(r, \frac{\lambda}{(\lambda + C) \binom{n}{2}} \big)$, of mean $\mu \ge 2\log(240C^2d)$, and
$\P( E_j(u)) = \P (N_j(u) \le 1) \le 2\mu e^{-\mu} \le 1/240C^2d$.  Moreover, for each fixed $j$,
the family of Boolean random variables $(\I_{E_j(u)})_{u \not=v}$ is negatively associated, so, for any set $D \subseteq V_n \setminus \{v\}$, the probability that all the variables
$(\I_{E_j(u)})_{u\in D}$ are equal to~1 is at most $(1/240C^2d)^{|D|}$.  Therefore, by Theorem~\ref{thm.ps97}, the probability that at least
$(n-1)/120C^2d$ of them are equal to~1 is at most $e^{-(n-1)/720C^2d}$.  (This application amounts to the fact that negatively associated random variables
satisfy the Chernoff-Hoeffding bounds.)  Thus the probability that at least $(n-1)/120Cd$ of the events $E_j(u)$ occur
($u\not=v$, $j\in \{1, \dots, C\}$) is at most $Ce^{-(n-1)/720C^2d}$.

We turn to $F_j(u)$, and first consider the case $j=1$.  The number of departures of calls incident with the node~$v$ over the interval $I_1$ of length~$r$
is dominated by a binomial random variable with parameters $\big( r, \frac{C(n-1)}{(\lambda +C) \binom{n}{2}} \big)$, of mean at most
$\frac{3C \log(240C^2d)}{\lambda} (n-1) \le \frac{n-1}{2400Cd}$.  The probability that this random variable is at least $(n-1)/1200Cd$ is therefore at most
$e^{-3(n-1)/2500Cd}$.  The number of events $F_1(u)$ that occur is at most twice the number of departures from links incident with $v$, so the probability
that more than $(n-1)/600Cd$ of the events $F_1(u)$ ($u \not= v$) occur is at most $e^{-(n-1)/400Cd}$.

Now fix $j \in \{2, \dots, C\}$, and a set $S \subseteq V_n \setminus \{v\}$.  The probability that all the random
variables $(F_j(u))_{u\in S}$ are equal to~1 is at most the probability that there are at least $\lceil j|S|/2 \rceil$ departures of calls on some link $uv$
with $u \in S$ during the interval $I_j$ of length~$jr$.  (Note that each departure may contribute to at most~2 of the variables.)  This probability is at most
$$
\binom{jr}{\lceil j|S|/2 \rceil} \left( \frac{C|S|}{(\lambda+C)\binom{n}{2}} \right)^{\lceil j|S|/2\rceil}
\le \left( \frac{2Cer}{(\lambda+C)\binom{n}{2}} \right)^{\lceil j|S|/2 \rceil}
$$
$$
\le \left( \frac{6eC}{\lambda} \log(240C^2d) \right)^{\lceil j|S|/2\rceil} \le \left( \frac{1}{480Cd} \right)^{j|S|/2}
\le \left( \frac{1}{480 Cd (j-1)^2} \right) ^{|S|}.
$$
Thus, for each $j \ge 2$, we may take $\delta = \frac{1}{480Cd (j-1)^2}$ in Theorem~\ref{thm.ps97}, so the probability that at least
$(n-1)/240Cd(j-1)^2$ of the random variables are equal to~1 is at most $e^{-(n-1)/1440Cd(j-1)^2}$.  Hence, with probability at least
$1- Ce^{-(n-1)/1440C^3d}$, at most $\frac{n-1}{240Cd} \sum_{j=2}^\infty 1/(j-1)^2 \le \frac{n-1}{150Cd}$ of the events $F_j(u)$ ($u\not=v$, $j=2, \dots, C$)
occur.

Thus, with probability at least $1-2Ce^{-(n-1)/1500C^3d}$, at most $(n-1)/60Cd$ of the events $E_j$ and $F_j$ occur, so there are at most $(n-1)/60Cd$ links
at $v$ that are not full at time $t$.
The result now follows; the probability that there is some node $v$ and some time $t \in [s,n^\kappa]$ such that there are fewer than
$(n-1)/60Cd$ links at $v$ that are not fully loaded at time $t$ is at most $2Cn^{\kappa +1}e^{-(n-1)/1500C^3d}$.
\end{proof}

Given two states $x$ and $y$, we call a link $uv$ {\em unbalanced} if $x(uv) \not= y(uv)$, and let $A_{x,y}$ be the set of unbalanced links.
In order to analyse the coupling in the case of high arrival rate, we use a specially tailored distance function.
We define $d(x,y)$ to be the sum of two contributions.  One is a multiple of the sum of the differences between numbers of indirect links on all
possible routes, and the other gives a contribution from each unbalanced link:
$$
d(x,y) = (4C+1) \sum_{uv,w} | x(uwv) - y(uwv) | + \sum_{uv \in A_{x,y}} \big(C - \min(x(uv), y(uv))\big).
$$
Note that $d(x,y) = 0$ if and only if $x=y$, and that $d(\cdot,\cdot)$ is symmetric.  Also, it is easy to see that $d(\cdot, \cdot)$ satisfies the
triangle inequality.  (Indeed, it suffices to check that, for each link $uv$, $\big(C - \min(x(uv), y(uv))\big) \I_{x(uv) \not= y(uv)}$
satisfies the triangle inequality $d(x,z) \le d(x,y) + d(y,z)$: this is immediate if any two of $x(uv)$, $y(uv)$ and $z(uv)$ are
equal, and otherwise two of the three terms are equal and the other is smaller.)  Thus $d(\cdot, \cdot)$ is a metric.

We also define a ``localised'' version of the distance.  Fix a node $v$, and define
\begin{eqnarray*}
d_v(x,y) &=& (4C+1) \left[ \sum_{u,w} | x(uwv) - y(uwv) | + \sum_{uw} |x(uvw) -y(uvw)| \right] \\
&&\mbox{} + \sum_{u} \big(C - \min(x(uv), y(uv))\big) \I_{x(uv) \not= y(uv)}.
\end{eqnarray*}

We couple $\widehat{X}$ and $\widehat{Y}$ as in Section~\ref{S:prelim}.  Recall in particular that we pair calls for departure occupying the same route, as
far as possible -- and that we call such pairs of calls {\em matched} -- and then do any further possible pairing for departures in an arbitrary way depending
only on the current states.

Before analysing how the distance behaves under a step of the coupling, we give a brief sketch, aimed at motivating the choice of distance function and the
strategy of the proof.  We imagine that $\lambda$ is very large, and so we expect that most links will be fully loaded (in the formal proof below, this is
reflected in the assumption that the starting states are in $R$).  Suppose first that two states $x$ and $y$ differ by one directly routed call, on link $uv$,
present in $x$ but not in $y$.  Coalescence between the states may occur on the departure of this unpaired call, but if $\lambda$ is large, then in the interim
it is quite likely that the differing loads on $uv$ between the two states will result in some arriving call either being routed on a different indirect route in the two states, or being accepted in one state and rejected in the other (below, we describe in detail the two types of {\em bad arrival} that may cause this).
The other possibility for coalescence is for the load $x(uv)$ on link $uv$ to be $C$, and an arriving call with endpoints $u$ and $v$ to be accepted
onto link $uv$ in $\widehat{Y}$, and rejected completely in $\widehat{X}$: we call such an arrival a {\em good arrival}; for the probability of such an arrival
to be suitably large, we need the network to be very full, so that once a call arrives for endpoints $u$ and $v$ that can be routed directly in $y$ but not $x$, it is quite unlikely for it to succeed in being routed indirectly.  If $x(uv)$ is less than $C$, then immediate coalescence at the next step is impossible:
however in this case, any call arriving with endpoints $u$ and $v$ will be routed directly in both states, and this moves us closer to our goal.  We represent
this in our distance function via the term $\big(C - \min(x(uv), y(uv))\big)$ for the unbalanced link $uv$, which decreases by~1 on such an arrival,
which we consider under the same umbrella as the good arrival that results in immediate coalescence.  A side-effect of this choice is that departures of
matched calls on link $uv$ ({\em bad departures of type~1}) increase the distance, but these are rare and their contribution to the expected change in
distance is relatively small.

Now suppose that the two states differ by the presence of one indirectly routed call, say on route $uwv$ in $x$.  Again, in our analysis we cannot
wait for this call to depart, as in the interim it will give rise to other unmatched indirect calls on ``bad arrivals''.  In this case, we move toward
coalescence in stages: we deem it to be progress for calls to arrive with endpoints $\{u,w\}$ or $\{w,v\}$, increasing the loads of the two links $uw$ and
$wv$, and ultimately providing extra direct calls (in the proof below, these are {\em covered} direct calls) in $\widehat{Y}$ on each of the two links via good
arrivals as before.  Once the two covered calls in $\widehat{Y}$ are in place, all links are balanced, so in the coupling every arrival is routed the same way in the two states; in particular, no bad arrival is possible.  When all links are balanced, departures of matched calls do not affect our distance function, so the only changes that contribute to the expected change in distance are departures of the unpaired calls.  The departure of either of the two covered direct calls in $\widehat{Y}$ increases the distance by at most $C$, but the departure of the unpaired indirect call on $uwv$ in $\widehat{X}$ decreases the distance by at least $2C+1$ (the term $(4C+1) | x(uwv) - y(uwv) |$ decreases by $4C+1$, but the departure of this call will also unbalance the links it occupies, contributing an increase in distance of up to $2C$): the choice of the coefficient $4C+1$ ensures that, in expectation, the distance goes down when the first
departure of an unpaired call occurs.

\begin{lemma} \label{lem.one-step}
(a) For all $t \ge 1$, under the coupling,
$$
\E \big( d(\widehat{X}_t,\widehat{Y}_t) \mid \cG_{t-1} \big) \le \Big( 1- \frac{1}{10 C \lambda \binom{n}{2}} \Big) d(\widehat{X}_{t-1},\widehat{Y}_{t-1}),
$$
on the event that $\widehat{X}_{t-1}$ and $\widehat{Y}_{t-1}$ are both in $R$.

\noindent
(b) For all $t \ge 1$, under the coupling,
$$
\E (d_v(\widehat{X}_t,\widehat{Y}_t) \mid \cG_{t-1} )
\le \Big( 1 - \frac{1}{10C\lambda\binom{n}{2}} \Big) d_v(\widehat{X}_{t-1},\widehat{Y}_{t-1}) + \frac{300d^2C}{n^3} d(\widehat{X}_{t-1},\widehat{Y}_{t-1}),
$$
on the event that $\widehat{X}_{t-1}$ and $\widehat{Y}_{t-1}$ are both in $R$.
\end{lemma}

\begin{proof}
(a) We fix states $x$ and $y$ in $R$, and consider
$$
\E \big (d(\widehat{X}_t,\widehat{Y}_t)-d(\widehat{X}_{t-1},\widehat{Y}_{t-1}) \mid \widehat{X}_{t-1} = x, \widehat{Y}_{t-1} = y \big).
$$
Let $a(x,y) = |A_{x,y}|$ be the number of unbalanced links.  Let $b(x,y) = \sum_{uv,w} |x(uwv) - y(uwv)|$ be the number
of unmatched indirect calls.

In the coupling, we call an unmatched direct call in one of the two states -- say on link $uv$ in state~$x$ -- {\em covered} if $x(uv) \le y(uv)$.  Note that, on each link $uv$, unmatched direct calls are present in at most one of the two states, and if these calls are covered then we can associate each of the covered
calls to an unmatched indirectly routed call occupying the link in the other state; each unmatched indirectly routed call is associated with at most two
covered calls, so the number $c(x,y)$ of covered calls is at most twice the number $b(x,y)$ of unmatched indirect calls.

We consider in turn each of the various ways in which the distance can change on the event at time~$t$.
First we consider arrivals.  An arrival can only change the distance if there is some unbalanced link $uv$ that is
inspected during the arrival (either because the endpoints of the call are $\{u,v\}$, or because an indirect route including the link $uv$ is considered).
In all other cases, the arriving call will be routed the same way in both $\widehat{X}$ and $\widehat{Y}$, onto some balanced link(s), and there will
be no change to the distance.  Moreover, if the call is for endpoints $\{u,v\}$ and $x(uv) = y(uv) < C$, then the distance will not change.
If $x(uv) = y(uv) = C$, then the distance can only change if either the call is routed differently in the two states, or if the call is assigned
to an indirect route using an unbalanced link; for either of these to occur, there has to be an indirect route considered
where one of the two links in the route is unbalanced -- say it has lower load in $y$ -- and the other link on the route is not fully loaded in
$y$.  We now analyse all the possibilities for an arrival that may change the distance, labelling them as ``good'' or ``bad'' according to whether they
decrease or increase the distance, respectively.

A {\em good arrival} is an arrival for a pair of nodes that are the endpoints of some unbalanced link $uv$ -- say with $x(uv) > y(uv)$ -- and such that each of
the links $uw_1, \dots, uw_d$ is fully loaded in $x$, where $w_1, \dots, w_d$ are the intermediate nodes considered.  We claim that the
distance decreases by~1 on a good arrival.  If $x(uv) < C$, then the arriving call is routed directly on $uv$ in both $\widehat{X}$ and~$\widehat{Y}$.  The only change to the distance is in the term $\big(C -\min(x(uv), y(uv))\big)$, which is decreased by~1.  If $x(uv) = C$, then the call is routed directly onto $uv$ in $\widehat{Y}$, and not routed at all in $\widehat{X}$, as all the routes considered are unavailable.  Hence $x(uv)$ is unchanged and $y(uv)$
is increased by~1, resulting in a decrease by~1 of $\big(C -\min(x(uv), y(uv))\big)\I_{x(uv) \not= y(uv)}$, while no other term
contributing to the distance is affected.  The probability of a good arrival, conditional on $\widehat{X}_{t-1}=x$ and $\widehat{Y}_{t-1}=y$, is at least
$$
\frac {\lambda a(x,y)}{(\lambda + C) \binom{n}{2}} \left( 1 - \frac{1}{60Cd} \right)^d \ge
\frac {\lambda a(x,y)}{(\lambda + C) \binom{n}{2}} \left( 1 - \frac{1}{60C} \right) \ge \frac{59}{60} \frac {\lambda a(x,y)}{(\lambda + C) \binom{n}{2}}.
$$
Here we used the fact that $x \in R$, so that the proportion of links at $u$ that are not fully loaded in $x$ is at most $1/60Cd$.

A {\em bad arrival of type~1} is an arrival for endpoints $\{ u,v\}$, where $uv$ is an unbalanced link -- again say with $x(uv) > y(uv)$ -- but such that at
least one of the links $uw_j$ ($j \in \{1,\dots,d\}$) inspected for indirect routing is not fully loaded.  In the case where $x(uv) = C$, such an arrival is routed directly onto the link $uv$ in $\widehat{Y}$, and may be routed onto the 2-link route $uw_jv$ in $\widehat{X}$.  In the worst case, this introduces an extra indirect route in $\widehat{X}$, adding $4C+1$ to the distance, and also causes both links $uw_j$ and $w_jv$ to become unbalanced,
giving a new contribution of up to $C$ to the distance for each of those two links.  On the other hand, the contribution from link
$uv$ decreases by~1, so the total increase in distance from a bad arrival of type~1 is at most $6C$.  The probability of a bad arrival of type~1,
conditional on $\widehat{X}_{t-1}=x$ and $\widehat{Y}_{t-1}=y$, is at most
$\frac {\lambda a(x,y)}{(\lambda + C) \binom{n}{2}} \frac{1}{60C}$,
as in the analysis for a good arrival.

A {\em bad arrival of type~2} is an arrival for endpoints $\{u, v\}$ where the link $uv$ is fully loaded in both $x$ and $y$, and where at least one indirect
route $uwv$ is considered that consists of an unbalanced link $uw$ -- say with $x(uw) > y(uw)$ -- and a link $wv$ that is not fully loaded in $y$.  On such an
arrival, it may be the case that different indirect routes are chosen in $\widehat{X}$ and in $\widehat{Y}$ (or an indirect route is chosen in $\widehat{Y}$,
and the call is rejected in $\widehat{X}$).  As in the analysis of a bad arrival of type~1, the introduction of a new indirect call in one state may increase the distance by up to $6C+1$, so the maximum increase in distance on a bad arrival of type~2 is $12C+2 \le 14C$.  The conditional probability of a bad arrival of
type~2 is at most
$$
2 a(x,y) \frac{n-1}{60Cd} d \frac{\lambda}{(\lambda + C)\binom{n}{2}} \frac{1}{n-2} \le \frac{\lambda a(x,y)}{(\lambda+ C)\binom{n}{2}} \frac{1}{20C},
$$
since there are at most $2a(x,y)$ choices of an unbalanced link $uw$, together with one of its ends $u$ to be an endpoint for the
arriving call, then at most $(n-1)/60Cd$ choices for a node $v$ such that $wv$ is not fully loaded in $y$ -- since $y \in R$ -- and $d$ choices for the
place where the indirect route $uwv$ is in the list of potential indirect routes.

The distance between $\widehat{X}$ and $\widehat{Y}$ can change on a departure for two reasons: either some matched pair of calls
departs, changing the load on some unbalanced link $uv$, or a pair of unmatched calls or a single unpaired call departs.  If an unmatched direct call, say in
$x$, departs, and the link it occupies has higher load in $x$ than in $y$, then the distance does not increase.  So the only case in which the distance
increases on the departure of an unmatched direct call is if the call is covered.  We shall see that the departure of an unmatched indirectly routed call
always decreases the distance.

A {\em bad departure of type 1} is a departure of a pair of matched calls using some unbalanced link(s) $uv$.  A bad departure can
increase the distance by at most~2 (if an indirect call departs and both links on the call are unbalanced).  The conditional probability of a bad departure
of type~1 is at most
$\frac{ C a(x,y)} {(\lambda + C) \binom{n}{2}}$,
as there are at most $C a(x,y)$ pairs of matched calls using unbalanced links.

A {\em bad departure of type 2} is a departure of a covered direct call on some link $uv$.  Such a departure may lead to an increase in distance of up to $C$,
either because the departure unbalances the link, creating a new contribution of $\big(C - \min(x(uv), y(uv)) + 1\big)$, or because the link is already
unbalanced and its contribution increases by~1.  The conditional probability of a bad departure of type~2 is equal to
$\frac{c(x,y)}{(\lambda + C) \binom{n}{2}}$.
It is possible for two bad departures of type~2 to happen together, if two covered direct calls in the different states are paired.  Such a pair of departures
leads to an increase of up to $2C$ in the distance, but accounts for 2 of the $c(x,y)$ covered direct calls; thus the pairing of these calls to depart
together gives the same contribution to the expected increase in distance as would be obtained by having the two departures occupy separate
``departure slots''.

A {\em good departure} is the departure of an unmatched indirect call on some route $uwv$.  This yields a decrease in distance of $4C+1$ due to
the reduction by~1 of $|x(uwv) - y(uwv)|$.  There may be an increase of up to $C$ in the contributions from each of the two links $uw$ and $vw$, but
overall there is a decrease in distance of at least~$2C+1$ on a good departure of type~2.  The conditional probability of a good departure is equal to
$\frac{b(x,y)}{(\lambda + C) \binom{n}{2}}$.
As above, two good departures may happen simultaneously, or indeed a good departure may be paired with a bad departure of type~2 in the
other state, but making such pairings does not affect the expected change in distance.

Summing the contributions to the expected change in distance from all the possible types of good or bad arrivals and departures, we have
\begin{eqnarray*}
\lefteqn{ \E \big( d(\widehat{X}_t,\widehat{Y}_t) - d(\widehat{X}_{t-1},\widehat{Y}_{t-1}) \mid \widehat{X}_{t-1} = x, \widehat{Y}_{t-1} = y \big) } \\
&\le& \frac{1}{(\lambda + C) \binom{n}{2}} \Big[ (-1)\frac{59}{60} \lambda a(x,y) + (6C) \frac{\lambda a(x,y)}{60C} + (14C) \frac{\lambda a(x,y)}{20C} \\
&&\mbox{} + (2) C a(x,y) + (C) c(x,y) + (-2C-1) b(x,y)\Big] \\
&\le& \frac{1}{(\lambda + C) \binom{n}{2}} \Big[ \lambda a(x,y) \Big( -\frac{59}{60} \!+\! \frac{1}{10} \!+\! \frac{7}{10} \!+\! \frac{2C}{\lambda} \Big) 
+ b(x,y) (2C \!-\! 2C \!-\! 1) \Big] \\
&\le& - \frac{1}{4C+1} \frac{1}{(\lambda + C) \binom{n}{2}} d(x,y).
\end{eqnarray*}
The final inequality follows since $d(x,y) \le C a(x,y) + (4C+1) b(x,y)$, and
$$
\frac{59}{60} - \frac{8}{10} -\frac{2C}{\lambda} \ge \frac{1}{10} \ge \frac{C}{(4C+1) \lambda}.
$$

Hence we have
\begin{eqnarray*}
\E \big( d(\widehat{X}_t,\widehat{Y}_t) \mid \cG_{t-1} \big) &\le& \left( 1- \frac{1}{4C+1} \frac{1}{(\lambda + C) \binom{n}{2}} \right) d(\widehat{X}_{t-1},\widehat{Y}_{t-1})\\
&\le& \left( 1 - \frac{1} {10C\lambda \binom{n}{2}} \right) d(\widehat{X}_{t-1},\widehat{Y}_{t-1}),
\end{eqnarray*}
on the event that $\widehat{X}_{t-1}$ and $\widehat{Y}_{t-1}$ are in $R$, as desired.

\medskip

\noindent
(b) We now bound the expected change $\Delta d_v = \E \big( d_v(\widehat{X}_t, \widehat{Y}_t) - d_v(\widehat{X}_{t-1}, \widehat{Y}_{t-1}) \big)$ in $d_v$,
conditional on $\widehat{X}_{t-1} = x$ and $\widehat{Y}_{t-1} = y$, where $x$ and $y$ are states in $R$.  Let $a_v(x,y)$ be the number of
unbalanced links $uv$ incident with $v$.  Let $b_v(x,y)$ be the number of unmatched indirect calls incident with $v$, and let $c_v(x,y)$ be the number of
covered direct calls in $x$ or $y$ on links incident with~$v$.  Note that, for the same reason as in~(a), $c_v(x,y) \le 2 b_v(x,y)$.

First, we bound the contribution to $\Delta d_v$ due to an arriving call inspecting an unbalanced link
$uw$ not incident with~$v$.  Such an arrival can lead to an increase in $d_v$ of at most $10C+2$ (if the call is for endpoints $\{u,v\}$, and is
routed via $w$ in one of $\widehat{X}$ and $\widehat{Y}$ but not the other, giving two extra calls on indirect routes including $v$, and potentially creating
two new unbalanced links at $v$).  The probability of such an arrival is at most
$$
\frac{(2d+1) a(x,y) \lambda}{(\lambda + C) \binom{n}{2}} \frac{d}{n-2} \le \frac{9d^2\lambda}{n} d(x,y) \frac{1}{(\lambda + C)\binom{n}{2}}.
$$
Note that $d_v$ only changes on a departure if the call uses a link incident with $v$ on which there is some contribution to $d_v(x,y)$,
either because the link is unbalanced or because it carries an unmatched indirectly routed call.

We now go through the various types of good and bad arrivals and departures identified in the proof of part~(a),
focussing on the contribution to $\Delta d_v$ due to differences between the loads in $x$ and $y$ on routes incident to~$v$.

A good arrival is an arrival with endpoints $\{ u,v\}$, where $uv$ is an unbalanced link incident to $v$, such that all the links $uw_1, \dots, uw_d$ are
fully loaded.  Such an arrival leads to a decrease by~1 in $d_v$, and has probability at least
$\frac{59}{60} \frac{\lambda a_v(x,y)}{(\lambda + C) \binom{n}{2}}$, provided $x$ and $y$ are in $R$.

A bad arrival of type 1 is an arrival with endpoints $\{ u,v\}$, where $uv$ is an unbalanced link incident to $v$, say with $y(uv) < x(uv)$, such that at
least one link $uw_j$ is not fully loaded in $y$.  Such an arrival leads to an increase in $d_v$ by at most~$5C$, and has probability at most
$\frac{\lambda a_v(x,y)}{(\lambda + C) \binom{n}{2}} \frac{1}{60C}$.

A bad arrival of type 2 can take two different forms.  In each, there are three nodes $u,v,w$, where the three links joining these nodes
consist of: one link that is fully loaded in both $x$ and $y$, one unbalanced link with, say, greater load in $x$, and a third link that is not fully
loaded in $y$.  The node $v$ is incident to the unbalanced link and one of the other two links.  The arrival is for the two endpoints of the fully loaded link,
and the third node is one of the intermediate nodes chosen.  A bad arrival of type~2 can increase $d_v$ by up to $10C+2$, and has probability at most
$$
a_v(x,y) \frac{2(n-1)}{60Cd} \frac{\lambda}{(\lambda + C)\binom{n}{2}} \frac{d}{n-2} \le \frac{\lambda a_v(x,y)}{(\lambda + C) \binom{n}{2}} \frac{1}{20C}.
$$

A bad departure of type 1 is a departure of matched calls on the same route, affecting one or two unbalanced links incident with $v$.
Such a departure increases $d_v$ by at most~2, and has probability at most $\frac{Ca_v(x,y)}{(\lambda + C) \binom{n}{2}}$.

A bad departure of type 2 is a departure of a covered direct call incident with $v$, say on link $uv$ in $x$.  Such a departure can lead to an increase in
$d_v$ of at most~$C$, and has probability at
most $\frac{c_v(x,y)}{(\lambda+C) \binom{m}{2}}  \le 2 \frac{b_v(x,y)}{(\lambda+C) \binom{m}{2}}$.

A good departure is a departure of an unmatched indirect call incident with $v$ in one state.  This departure reduces $d_v$ by at
least $(4C+1) - 2C = 2C+1$, and the probability of such a departure is at least $\frac{b_v(x,y)}{(\lambda+C) \binom{m}{2}}$.

Combining these contributions, we find that, for $x,y \in R$,
\begin{eqnarray*}
\lefteqn{ \E (d_v(\widehat{X}_t,\widehat{Y}_t) - d_v(\widehat{X}_{t-1},\widehat{Y}_{t-1}) \mid \widehat{X}_{t-1} = x, \widehat{Y}_{t-1} = y) } \\
&\le& \frac{1}{(\lambda + C) \binom{n}{2}} \Big[ a_v(x,y) \Big( (-1) \frac{59\lambda}{60} + (5C)\frac{\lambda}{60C} + (10C+2)\frac{\lambda}{20C} + (2)C \Big) \\
&&\mbox{} + (C) c_v(x,y) - (2C+1) b_v(x,y) + (10C+2) \frac{9d^2\lambda}{n} d(x,y) \Big] \\
&\le& \frac{1}{(\lambda+C)\binom{n}{2}} \Big[ - \frac{\lambda}{4} a_v(x,y) - b_v(x,y) + \frac{108d^2C\lambda}{n} d(x,y) \Big] \\
&\le& \frac{1}{(\lambda+C)\binom{n}{2}} \Big[ - \frac{1}{4C+1} d_v(x,y) + \frac{108d^2C\lambda}{n} d(x,y) \Big] \\
&\le&  - \frac{1}{10C\lambda\binom{n}{2}} d_v(x,y) + \frac{300d^2C}{n^3} d(x,y),
\end{eqnarray*}
which gives the desired result.
\end{proof}

We now use an inductive argument to establish an analogue to Lemma~\ref{lem.coupling} for this regime.
One issue we face in applying Lemma~\ref{lem.one-step} is that the processes may not start in the set $R$ (or they may exit the set straight away
with reasonable probability); all we know is that they remain in $R$ with high probability after the ``burn-in'' period of $s$ steps.  To deal with this, we apply Lemma~\ref{lem.coupling} to show that, if the processes start close together, then they do not get
too far apart during the ``burn-in'' period.  Recall that $s = 3 \binom{n}{2} \frac{C(\lambda+C)}{\lambda} \log(240C^2d)$.

\begin{lemma} \label{lem.coupling2}
Suppose $\lambda \ge \lambda_1(d,C)$, and $\kappa > 0$.
Then for all sufficiently large $n$, and all $t \in [s,n^\kappa]$, we have
\begin{eqnarray}
\E (\|\widehat{X}_t-\widehat{Y}_t\|_1 \mid \cG_0)
&\le& 5C e^{40dC\log(240C^2d)} \left( 1 - \frac{1}{10C\lambda \binom{n}{2}} \right)^{t-s}
\|\widehat{X}_0-\widehat{Y}_0\|_1 \nonumber \\
&&\mbox{} + 2C^2n^{\kappa +3}e^{-(n-1)/1500C^3d}, \label{eq.rapid-mixing}
\end{eqnarray}
and, for each node $v$,
\begin{eqnarray*}
\lefteqn{\E (\|\widehat{X}_t-\widehat{Y}_t \|_v \mid \cG_0)\le 5C e^{40dC^2\log(240C^2d)} \left( 1 - \frac{1}{10C\lambda \binom{n}{2}} \right)^{t-s}} \\
&&\mbox{} \times \left( \|\widehat{X}_0-\widehat{Y}_0\|_v + \left(\frac{75d^2C\log(240C^2d)} {n} +
\frac{400d^2C (t-s)}{n^3}\right) \|\widehat{X}_0-\widehat{Y}_0\| \right) \\
&&\mbox{} + 4C^2n^{\kappa +2}e^{-(n-1)/1500C^3d}.
\end{eqnarray*}
\end{lemma}

\begin{proof}
From Lemma~\ref{lem.coupling} with $t=s$, we obtain that
\begin{equation}
\E ( \|\widehat{X}_s - \widehat{Y}_s \|_1 \mid \cG_0) \le \left( 1 + \frac{8d+4}{\binom{n}{2}} \right)^s \| \widehat{X}_0 - \widehat{Y}_0 \|_1
\le e^{40dC \log(240C^2d)} \| \widehat{X}_0 - \widehat{Y}_0 \|_1, \label{eq.burn-in1}
\end{equation}
and, for each node $v$,
\begin{eqnarray}
\lefteqn{ \E ( \|\widehat{X}_s - \widehat{Y}_s \|_v \mid \cG_0)} \label{eq.burn-in2} \\
&\le& \left( 1 + \frac{8d+4}{\binom{n}{2}} \right)^s \left( \| \widehat{X}_0 - \widehat{Y}_0 \|_v + \frac{50d^2 \lambda s}{(\lambda+C)n^3}
\| \widehat{X}_0 - \widehat{Y}_0 \|_1 \right) \nonumber \\
&\le& e^{40dC \log(240C^2d)} \left( \| \widehat{X}_0 - \widehat{Y}_0 \|_v + \frac{75d^2 C \log(240C^2d)}{n}
\| \widehat{X}_0 - \widehat{Y}_0 \|_1 \right). \nonumber 
\end{eqnarray}

For $t \ge s$, let $B_t$ be the event that $\widehat{X}_r$ and $\widehat{Y}_r$ are both in $R$ for all $r \in [s,t]$.  Then we have from
Lemma~\ref{lem.one-step} that, for $t \ge s$,
\begin{eqnarray*}
\E (d(\widehat{X}_{t+1},\widehat{Y}_{t+1}) \I_{B_{t+1}} \mid \cG_t) &\le& \Big( 1 - \frac{1}{10C\lambda \binom{n}{2}} \Big)
d(\widehat{X}_t,\widehat{Y}_t) \I_{B_t}; \\
\E (d_v(\widehat{X}_{t+1},\widehat{Y}_{t+1}) \I_{B_{t+1}} \mid \cG_t) &\le& \Big( 1 - \frac{1}{10C \lambda \binom{n}{2}} \Big)
d_v(\widehat{X}_t,\widehat{Y}_t) \I_{B_t} \\
&&\mbox{} + \frac{300 d^2 C}{n^3} d(\widehat{X}_t,\widehat{Y}_t) \I_{B_t}.
\end{eqnarray*}
From these it follows, as in the proof of Lemma~\ref{lem.coupling}, that
$$
\E (d(\widehat{X}_t,\widehat{Y}_t) \I_{B_t} \mid \cG_s) \le \left( 1 - \frac{1}{10C \lambda \binom{n}{2}} \right)^{t-s} d(\widehat{X}_s,\widehat{Y}_s);
$$
\begin{eqnarray*}
\lefteqn{ \E (d_v(\widehat{X}_t,\widehat{Y}_t) \I_{B_t} \mid \cG_s) } \\
&\le& \left( 1 - \frac{1}{10C\lambda \binom{n}{2}} \right)^{t-s} d_v(\widehat{X}_s,\widehat{Y}_s) \\
&&\mbox{} + \frac{300d^2C (t-s)}{n^3} d(\widehat{X}_s,\widehat{Y}_s) \left( 1 - \frac{1}{10C\lambda \binom{n}{2}} \right)^{t-s-1} \\
&\le& \left( 1 - \frac{1}{10C\lambda \binom{n}{2}} \right)^{t-s} \left( d_v(\widehat{X}_s,\widehat{Y}_s) + \frac{400d^2C (t-s)}{n^3} d(\widehat{X}_s,\widehat{Y}_s)\right),
\end{eqnarray*}
provided $n$ is sufficiently large.
Using also that $\| x-y\|_1 \le d(x,y) \le 5C \| x-y\|_1$ and $\| x-y\|_v \le d_v(x,y) \le 5C \| x-y\|_v$ for each $v$, we now have
$$
\E (\| \widehat{X}_t-\widehat{Y}_t)\|_1 \I_{B_t} \mid \cG_s) \le 5C \left( 1 - \frac{1}{10C \lambda \binom{n}{2}} \right)^{t-s} \|\widehat{X}_s-\widehat{Y}_s\|_1;
$$
\begin{eqnarray*}
\lefteqn{ \E (\|(\widehat{X}_t-\widehat{Y}_t)\|_v \I_{B_t} \mid \cG_s) } \\
&\le& 5C \left( 1 - \frac{1}{10C\lambda \binom{n}{2}} \right)^{t-s} \left( \|\widehat{X}_s-\widehat{Y}_s\|_v
+ \frac{400d^2C (t-s)}{n^3} \|\widehat{X}_s-\widehat{Y}_s\|_1\right),
\end{eqnarray*}
provided $n$ is sufficiently large.
Combining these inequalities with (\ref{eq.burn-in1}) and (\ref{eq.burn-in2}) gives us that, for $s\le t \le n^\kappa$, provided $n$ is
sufficiently large,
$$
\E (\|\widehat{X}_t-\widehat{Y}_t\|_1 \I_{B_t} \mid \cG_0) \le 5C e^{40dC\log(240C^2d)} \left( 1 - \frac{1}{10C\lambda \binom{n}{2}} \right)^{t-s}
\|\widehat{X}_0-\widehat{Y}_0\|_1;
$$
\begin{eqnarray*}
\lefteqn{\E (d_v(\widehat{X}_t-\widehat{Y}_t) \I_{B_t} \mid \cG_0) \le 5C e^{40dC\log(240C^2d)} \left( 1 - \frac{1}{10C\lambda \binom{n}{2}} \right)^{t-s}} \\
&&\mbox{} \times \left( \|\widehat{X}_0-\widehat{Y}_0\|_v + \left(\frac{75d^2C\log(240C^2d)} {n} +
\frac{400d^2C (t-s)}{n^3}\right) \|\widehat{X}_0-\widehat{Y}_0\|_1 \right).
\end{eqnarray*}

Noting that $\|x-y\|_1 \le C\binom{n}{2}$ and $\|x-y\|_v \le Cn$ for any states $x$ and $y$, we also have from Lemma~\ref{lem.R}
that, for $s \le t \le n^\kappa$,
\begin{eqnarray*}
\E (\|\widehat{X}_t-\widehat{Y}_t\|_1 \I_{\overline{B_t}} \mid \cG_0) &\le& 2C^2n^{\kappa +3}e^{-(n-1)/1500C^3d}; \\
\E (\|\widehat{X}_t-\widehat{Y}_t\|_v \I_{\overline{B_t}} \mid \cG_0) &\le& 4C^2n^{\kappa +2}e^{-(n-1)/1500C^3d},
\end{eqnarray*}
and the result follows.
\end{proof}

It follows from (\ref{eq.burn-in1}) and (\ref{eq.rapid-mixing}) that, provided $\lambda\ge \lambda_1(d,C)$, there are constants
$K$ and $\varepsilon > 0$ such that, for $n$ sufficiently large and $0 \le t \le n^{5/2}$, 
\begin{equation} \label{eq.r-m}
\E (\| \widehat{X}_t - \widehat{Y}_t \|_1 \mid \cG_0) \le K e^{-\varepsilon t/n^2} \| \widehat{X}_0 - \widehat{Y}_0 \|_1.
\end{equation}
Theorem~\ref{thm.low-high}(1), for $\lambda \ge \lambda_1(d,C)$, then follows.

Exactly as in the previous section, we also obtain concentration of measure for the discrete jump chain starting in a fixed state, for
$\lambda \ge \lambda_1(d,C)$.

\begin{lemma} \label{lem.concentration2}
Suppose $\lambda \ge \lambda_1(d,C)$.
For any nodes $u$ and $v$, let $f$ be a $(u,v)$-Lipschitz function.
Then, for all sufficiently large $n$, all $a > 0$, all $t\le n^\kappa$, and all initial states $x_0$,
$$
\P_{x_0}\Big ( |f(\widehat{X}_t)-\E_{x_0} [f(\widehat{X}_t)] |\ge a \Big ) \le 2\exp \left(\frac{-a^2}{(2 \lambda n + a)\exp(1000dC^2\log(240C^2d))}\right).
$$
\end{lemma}

The proof of the lemma is in essence the same as for Lemma~\ref{lem.concentration}.  We use the bounds on $\alpha_i^2$ derived above, where
we have a different form of the bound according to whether $t\le s$ or $s<t \le n^\kappa$.

Having established this lemma, we proceed exactly as before.  We have concentration of measure for the discrete jump chain started from a fixed state,
and we have rapid mixing of the process to equilibrium.  From these, it follows that we have concentration of measure in equilibrium
for any $(u,v)$-Lipschitz function $f$: 
$$
\P ( |f(\widehat{Y}_t) - \E f(\widehat{Y}_t)| > 2a ) \le 3 \exp\left( \frac{-a^2}{(2 \lambda n + a)\exp(1000dC^2\log(240C^2d))}\right),
$$
where $\widehat{Y}_t$ denotes a copy of the discrete jump chain in equilibrium.  Applying the result to $f = f_{v,j}$ yields
Theorem~\ref{thm.low-high}(2) in the case $\lambda \ge \lambda_1(d,C)$.

For $a = \frac14 \sqrt n \log n$, we obtain, for sufficiently large~$n$, 
\begin{equation} \label{eq.conc2}
\P ( |f(\widehat{Y}_t) - \E f(\widehat{Y}_t)| > \frac12 \sqrt n \log n ) \le 3 e^{-\delta \log^2 n},
\end{equation}
where we may take $\delta$ equal to $32 \lambda \exp(1000dC^2 \log(240C^2d))$.


\section{The process in equilibrium}

\label{S:five}

In this section, as earlier, we use $\widehat{Z}$ to denote a copy of the process in equilibrium, and $\pi$ to denote the equilibrium distribution of the
process.
Given a link $uv$ and an integer $j \in \{0, \ldots, C\}$, let $\I_{uv}^j: S \to \{0,1\}$ be defined by $\I_{uv}^j (x) =1$ if $x(uv)=j$
and $\I_{uv}^j (x) = 0$ otherwise.  Note that
$\I_{uv}^j = \I_{vu}^j$ and that $\I_{vv}^j$ is identically $0$ for each $v$ and $j$.
Also, $\I_{uv}^{\le j}$ is the indicator function of the event that link $uv$ has load at most $j$.

Let $P$ be the transition matrix of the Markov chain $\widehat{X}$.  For a function $f: S \to \R$, define $(Pf): S \to \R$
by $(Pf)(x) = \sum_y P(x,y) f(y)$.  
By standard theory of Markov chains, for each $t \ge 0$, each $v \in V_n$ and each
$j \in \{0, \ldots, C\}$,
\begin{eqnarray*}
\E_\pi [(P f_{v,j})(\widehat{Z}_t) - f_{v,j} (\widehat{Z}_t)] = 0.
\end{eqnarray*}

For $x \in S$, we have
\begin{eqnarray}
(P f_{v,0})(x) - f_{v,0}(x) & = & \frac{1}{(\lambda\!+\!C)\binom{n}{2}} \big( - \lambda f_{v,0}(x)- \lambda g_{v,0}(x) + f_{v,1}(x) \big), \nonumber \\
(P f_{v,j})(x) - f_{v,j}(x) & = & \label{eq.stat} \\
\lefteqn{\frac{1}{(\lambda\!+\!C)\binom{n}{2}} \big( \lambda f_{v,j-1}(x) - \lambda f_{v,j}(x) + \lambda  g_{v,j-1}(x) - \lambda g_{v,j}(x) \qquad } \nonumber \\
&& \mbox{} - j f_{v,j}(x) + (j+1) f_{v,j+1}(x)\big) \quad (0<j<C), \nonumber\\
(P f_{v,C})(x) - f_{v,C}(x) & = & \frac{1}{(\lambda\!+\!C)\binom{n}{2}} \big( \lambda f_{v,C-1}(x)+\lambda g_{v,C-1}(x)  - C f_{v,C}(x) \big), \nonumber
\end{eqnarray}
where the $g_{v,j}$, representing contributions due to alternatively routed arrivals with one end $v$, are given, for $j=0, \ldots, C-1$, by:
\begin{eqnarray}
\frac{1}{(n-2)^d}\Bigg[ \sum_{r=1}^d\sum_{u,{\bf w}} \I_{uv}^C\I_{vw_r}^{j} \I_{uw_r}^{\le j} \prod_{s=1}^{r-1} (1-\I_{uw_s}^{\le j}\I_{vw_s}^{\le j})
\prod_{s=r+1}^d (1-\I_{uw_s}^{\le j-1}\I_{vw_s}^{\le j-1})\nonumber \\
+ \sum_{r=1}^d \sum_{u,{\bf w}} \I_{uv}^{C}\I_{vw_r}^j \sum_{i=j+1}^{C-1} \I_{uw_r}^{i}\prod_{s=1}^{r-1} (1-\I_{uw_s}^{\le i}\I_{vw_s}^{\le i})
\prod_{s=r+1}^d (1-\I_{uw_s}^{\le i-1}\I_{vw_s}^{\le i-1}) \nonumber \\
+ \sum_{r=1}^d \sum_{u,v',{\bf w}_r}  \I_{uv'}^{C}\I_{uv}^{j} \I_{v'v}^{\le j} \prod_{s=1}^{r-1} (1-\I_{uw_s}^{\le j}\I_{v'w_s}^{\le j})
\prod_{s=r+1}^d (1-\I_{uw_s}^{\le j-1}\I_{v'w_s}^{\le j-1}) \nonumber \\
+ \!\sum_{r=1}^d\sum_{u,v',{\bf w}_r} \I_{uv'}^{C}\I_{uv}^{j} \sum_{i=j+1}^{C-1} \I_{v'v}^{i}\prod_{s=1}^{r-1} \!(1-\I_{uw_s}^{\le i}\I_{v'w_s}^{\le i})
\!\prod_{s=r+1}^d \! (1-\I_{uw_s}^{\le i-1}\I_{v'w_s}^{\le i-1})\Bigg]. \label{eq-g}
\end{eqnarray}
Here, $\sum_{u,{\bf w}}$ denotes the sum over all $u \not = v$, and over all $w_1, \ldots, w_d$ such that each $w_r \not = u,v$, and
$\sum_{u,v',{\bf w}_r}$ denotes the sum over all $u \not = v$, $v' \not = u,v$ and over all
$w_1,\ldots, w_{r-1},w_{r+1}, \ldots, w_d$ such that each $w_j \not = u,v'$.

Recall that $\Delta^{C+1} = \{ \xi \in [0,1]^{C+1} : \sum_{j=0}^C \xi(j) = 1 \}$.  Define the vector $\zeta \in \Delta^{C+1}$ to have coordinates
$\zeta(j) = \frac{1}{n-1} \E_\pi f_{v,j}(\widehat{Z}_t)$, for $j = 0, \dots, C$.  Note that, by symmetry and stationarity, $\zeta$ is independent of $v$
and~$t$.

Taking expectations in (\ref{eq.stat}), we now have, for all $v$ and $t$,
\begin{eqnarray*}
0 &=& - \lambda \zeta(0) - \frac{\lambda}{n-1} \E_\pi g_{v,0} (\widehat{Z}_t) + \zeta(1), \\
0 &=& \lambda \zeta(j-1) - \lambda \zeta(j) + \frac{\lambda}{n-1} \E_\pi g_{v,j-1}(\widehat{Z}_t) - \frac{\lambda}{n-1} \E g_{v,j}(\widehat{Z}_t) \\
&&\mbox{} - j \zeta(j) + (j+1) \zeta(j+1) \qquad (0< j < C), \\
0 &=& \lambda \zeta(C-1) + \frac{\lambda}{n-1} \E_\pi g_{v,C-1}(\widehat{Z}_t) - C \zeta(C).
\end{eqnarray*}
Summing these equations over $j=0, \dots, k$ yields, for $k=0, \dots, C-1$,
\begin{equation} \label{eq.fixed-pt}
0 = - \lambda \zeta(k) - \frac{\lambda}{n-1} \E_\pi g_{v,k}(\widehat{Z}_t) + (k+1) \zeta(k+1).
\end{equation}

We claim that each $\frac{1}{n-1} \E_\pi g_{v,j}(\widehat{Z}_t)$ is close to $g_j(\zeta)$, where the $g_j$ are as defined in (\ref{eq.G}).
We state two properties of the process in equilibrium, whose proofs are practically identical to two proofs in~\cite{l12}, except that
they use the concentration of measure results from the previous sections.

For distinct nodes $u$ and $v$, and $j,k \in \{ 0, \dots, C\}$, we define
\begin{eqnarray*}
\phi^1_{u,v,j,k} &=& \frac{1}{n-2}\sum_{w} \I_{uw}^j \I_{vw}^{k} - \frac{1}{(n-2)^2}\sum_{w \not = u,v} \I_{uw}^{j} \sum_{w' \not = u,v}\I_{vw'}^{k}; \\
\phi^2_{u,v,j} & = & \frac{1}{n-2} (f_{u,j} - f_{v,j}).
\end{eqnarray*}
Then we set $\phi^1 = \max_{u,v,j,k} |\phi^1_{u,v,j,k}|$ and $\phi^2 = \max_{u,v,j} |\phi^2_{u,v,j}|$, where
the maximisations are over distinct nodes $u$ and $v$ and, where appropriate, $j,k \in \{ 0, \dots, C\}$.
Finally, let $\widetilde \phi = \max (\phi^1, \phi^2 )$.  The function $\widetilde{\phi}$ was used in~\cite{l12} as a measure of how``uniform'' a state is.
We shall show that the expected value of $\widetilde{\phi}(\widehat{Z}_t)$ is small.  
For a state to have a low value of $\phi_1$ means that, for each
fixed $u$ and $v$, if we choose a uniformly random ``intermediate node'' $w$, then the loads on the two links $uw$ and $wv$ are nearly independent.
To have a small value of $\phi_2$ means that all nodes have similar distributions of loads on the links incident with them.  In a state where $\widetilde{\phi}$
is small, the complex functions $g_{v,j}$ above can be ``approximately factorised'', so that they can be approximated by products of functions similar to
the $f_{v,j}$.  Then, applying our concentration inequalities, Lemmas~\ref{lem.concentration} and~\ref{lem.concentration2}, to
various functions $h$, we can show that the expectation of the product is close to the product of the expectations.  This phase of the proof is almost
identical to the proof of Lemma~6.2 in~\cite{l12} -- indeed, that proof goes through whenever the distribution of states satisfies a concentration
inequality of the same type as ours.  As the argument is somewhat lengthy, we omit the proof.


\begin{lemma} \label{lem.gen-exp}
Suppose that either $\lambda < \lambda_0(d)$ or $\lambda\ge \lambda_1(d,C)$.
For all sufficiently large $n$, for all $v \in V_n$ and $j \in \{0, \dots, C\}$,
$$
\Big | \E_\pi [g_{v,j}(\widehat{Z}_t)] - (n-1)g_j (\zeta) \Big | \le 8d^2 (C+1)^3 n \E_\pi [\widetilde{\phi} (\widehat{Z}_t)] + 16 d^2 (C+1) \sqrt{n} \log n.
$$
\end{lemma}

The only real difference between the proof of Lemma~\ref{lem.gen-exp} above and that of Lemma~6.2 in~\cite{l12} is that we use Lemma~\ref{lem.concentration}
or Lemma~\ref{lem.concentration2} in place of the concentration of measure result used in~\cite{l12}, which applied to a process starting in a fixed state,
and running for a bounded period.

It remains to bound $\E_\pi [\widetilde{\phi}]$.  The proof of the following lemma is identical to the argument in Section~8 of~\cite{l12}.  Indeed, that
proof goes through for any distribution of states that respects the symmetry of the complete graph, and is such that $(u,v)$-Lipschitz functions are
well-concentrated.

\begin{lemma} \label{lem.phi}
Suppose that either $\lambda < \lambda_0(d)$ or $\lambda\ge \lambda_1(d,C)$.
Then, for sufficiently large $n$,
$$
\E_\pi \widetilde{\phi}(\widehat{Z}_t) \le 4 \frac{\log n}{\sqrt n}.
$$
\end{lemma}

Combining Lemmas~\ref{lem.gen-exp} and~\ref{lem.phi}, we obtain that, if either $\lambda< \lambda_0(d)$ or $\lambda \ge \lambda_1(d,C)$, then 
for all sufficiently large $n$, for all $v \in V_n$ and $j \in \{0, \dots, C\}$,
\begin{equation*} 
\Big | \E_\pi [g_{v,j}(\widehat{Z}_t)] - (n-1)g_j (\zeta) \Big | \le 36 d^2 (C+1)^3 \sqrt{n} \log n.
\end{equation*}
Combining this with (\ref{eq.fixed-pt}) gives that, for $j =0, \dots, C-1$,
\begin{equation} \label{eq.bound}
\left| - \lambda \zeta(j) - \lambda g_j(\zeta) + (j+1) \zeta(j+1) \right| \le 40 \lambda d^2 (C+1)^3 \frac{\log n}{\sqrt n}.
\end{equation}

\section{Approximate solutions to $F(\eta)=0$}

The aim of this section is to prove parts~(3) and~(4) of Theorem~\ref{thm.low-high}, showing that, in either of the two regimes we are considering,
the equation $F(\eta)=0$ has a unique solution, and any approximate solution, i.e., a vector $\zeta$ satisfying (\ref{eq.bound}), lies close to this
solution.

Luczak~\cite{l12} shows, for each $k\in \{0, \dots, C\}$ and each pair $\xi,\eta$ in $\Delta^{C+1}$, that 
$|g_k(\xi) - g_k(\eta)| \le 3d^2 (C+1)^2 \|\xi-\eta\|_\infty$;
we begin this section by obtaining a sharper version of this Lipschitz-like inequality, which can be sharpened further in the two regimes we study.

For $0 \le b \le a \le 1$, we set
\begin{eqnarray*}
H[a,b] &=& (a-b) \sum_{r=1}^d (1-a^2)^{r-1} (1-b^2)^{d-r},
\end{eqnarray*}
so that, for $\xi \in \Delta^{C+1}$, and $k \in \{ 0, \dots, C-1\}$,
$$
g_k(\xi) = 2 \xi(C) \xi(\le k) H[\xi(\le\! k),\xi(\le\! k\!-\!1)] + 2 \xi(C) \xi(k) \sum_{i=k+1}^{C-1} H[\xi(\le\! i), \xi(\le\! i\!-\!1)];
$$
here we used the fact that $\xi(j) = \xi(\le\! j) - \xi(\le\! j\!-\!1)$ for each $j$.
Observe that $H[a,b] \le d (a-b)$.  We also see that
\begin{eqnarray*}
\left| \frac{\partial H[a,b]}{\partial a}\right|  &=& \Big| \sum_{r=1}^d (1-a^2)^{r-1} (1-b^2)^{d-r} \\
&&\mbox{} - 2a(a-b) \sum_{r=2}^d (r-1) (1-a^2)^{r-2} (1-b^2)^{d-r} \Big| \\
&\le& d + 2 \sum_{r=2}^d (r-1) a^2 (1-a^2)^{r-2} \le 3d,
\end{eqnarray*}
and similarly $\left| \frac{\partial H[a,b]}{\partial b}\right| \le 3d$.
Therefore, whenever $0 \le b \le a \le 1$ and $0\le b'\le a' \le 1$, we have
$\big| H[a,b] - H[a',b'] \big| \le 3d \big( |a-a'| + |b-b'| \big)$.

We write, for any two vectors $\xi, \eta \in \Delta^{C+1}$, and any $k \in \{ 0, \dots, C-1\}$,
\begin{eqnarray*}
\lefteqn{ g_k(\xi) - g_k(\eta) } \\
&=& 2[\xi(C)-\eta(C)] \Big\{ \xi(\le\! k) H[\xi(\le\! k),\xi(\le\! k\!-\!1)]  \\
&&\mbox{} + \xi(k) \sum_{i=k+1}^{C-1} H[\xi(\le\! i), \xi(\le\! i\!-\!1)] \Big\} \\
&&\mbox{} + 2\eta(C) [\xi(\le\! k) - \eta(\le\! k)] H[\xi(\le\! k),\xi(\le\! k\!-\!1)] \\
&&\mbox{} + 2\eta(C) \eta(\le\! k) \left[ H[\xi(\le\! k),\xi(\le\! k\!-\!1)] - H[\eta(\le\! k),\eta(\le\! k\!-\!1)] \right] \\
&&\mbox{} + 2\eta(C) [\xi(k) - \eta(k)] \sum_{i=k+1}^{C-1} H[\xi(\le\! i), \xi(\le\! i\!-\!1)] \\
&&\mbox{} + 2\eta(C) \eta(k) \sum_{i=k+1}^{C-1} \left[ H[\xi(\le\! i), \xi(\le\! i\!-\!1)] - H[\eta(\le\! i), \eta(\le\! i\!-\!1)] \right].
\end{eqnarray*}
So, using our bounds on $H[a,b]$ and $\big| H[a,b] - H[a',b'] \big|$, we have
\begin{eqnarray*}
\lefteqn{ | g_k(\xi) - g_k(\eta) |} \\
&\le& 2 |\xi(C)-\eta(C)| d \left\{ \xi(\le\! k) \xi(k) + \xi(k) \sum_{i=k+1}^{C-1} \xi(i) \right\} \\
&&\mbox{} + 2 \eta(C) | \xi(\le\! k) - \eta(\le\! k) | d \xi(k) \\
&&\mbox{} + 2 \eta(C) \eta(\le\! k) 3d \big( |\xi(\le\! k) - \eta(\le\! k)| + |\xi(\le\! k\!-\!1) - \eta(\le\! k\!-\!1)| \big) \\
&&\mbox{} + 2 \eta(C) | \xi(k) - \eta(k) | d \sum_{i=k+1}^{C-1} \xi(i) \\
&&\mbox{} + 2 \eta(C) \eta(k) 3d \sum_{i=k+1}^{C-1} \big( |\xi(\le\! i) - \eta(\le\! i)| + |\xi(\le\! i\!-\!1) - \eta(\le\! i\!-\!1)| \big) \\
&\le& d \xi(k) \| \xi - \eta \|_1 \xi(\le\! C\!-\!1) + \eta(C) d \Big[ \xi(k) \|\xi - \eta\|_1 + 6 \|\xi - \eta\|_1 \eta(\le\! k) \\
&&\mbox{} + 2 | \xi(k) - \eta(k) | \xi(\le\! C\!-\!1) + 6C \eta(k) \| \xi-\eta\|_1 \Big].
\end{eqnarray*}
Summing over $k =0, \dots, C-1$, we have that, for any vectors $\xi, \eta \in \Delta^{C+1}$,
\begin{eqnarray}
\lefteqn{\| g(\xi) - g(\eta) \|_1} \nonumber \\
&\le& d \|\xi-\eta\|_1 \xi(\le\! C\!-\!1)^2 + \eta(C) d \Big( \|\xi-\eta\|_1 \xi(\le\! C\!-\!1) + 6 C \|\xi-\eta\|_1 \eta(\le\! C\!-\!1) \nonumber \\
&&\mbox{} +2 \|\xi-\eta\|_1 \xi(\le\! C\!-\!1) + 6C \|\xi-\eta\|_1 \eta(\le\! C\!-\!1) \Big) \nonumber \\
&\le& \big( 1 + \eta(C) (12C + 3) \big) d \| \xi-\eta\|_1 \max(\xi(\le\! C\!-\!1), \eta(\le\! C\!-\!1)). \label{eq.lip}
\end{eqnarray}


At this point, we separate the calculations for small and large $\lambda$.  Suppose first that $\lambda < \lambda_0(d) = 1/(8d+4)$.
In this case, we simplify the bound (\ref{eq.lip}) above to obtain, for any $\xi$ and $\eta$ in $\Delta^{C+1}$, that
$$
\| g(\xi) - g(\eta) \|_1 \le \big( 1 + \eta(C) (12C + 3) \big) d \| \xi-\eta\|_1.
$$
Suppose also that $\eta$ satisfies the fixed-point equation
\begin{equation} \label{eq.fixed-point}
- \lambda \eta(j) - \lambda g_j(\eta) + (j+1) \eta(j+1) = 0 \quad (j =0, \dots, C-1).
\end{equation}
Noting that $g_j(\eta) \le 2d \eta(j) \eta(C)$ for all $j$, we obtain
$$
\eta(j+1) \le \frac{\eta(j) \lambda ( 1 + 2d \eta(C) )}{j+1}, \mbox{ and hence } \eta(j) \le \frac{\big(\lambda (1+ 2d \eta(C)\big)^j}{j!},
$$
for each $j$.
For $\lambda \le \lambda_0(d) = 1/(8d+4)$, we apply this inequality to obtain
$\eta(C) \le 1/4^C C! \le 1/4$,
and then apply it again to obtain
$$
(12C + 3) \eta(C) \le \frac{12C + 3}{8^C C!} < 2.
$$
Thus, for any $\xi \in \Delta^{C+1}$, and any $\eta \in \Delta^{C+1}$ satisfying (\ref{eq.fixed-point}), we have
$$
\| g(\xi) - g(\eta) \|_1 \le 3d \| \xi-\eta \|_1,
$$
provided $\lambda \le \lambda_0(d)$.

Now let $\zeta$ be any vector in $\Delta^{C+1}$ satisfying (\ref{eq.bound}), and $\eta$ any vector in $\Delta^{C+1}$ satisfying (\ref{eq.fixed-point}).
We have, for $j=0, \dots, C-1$,
$$
(j+1) | \zeta(j+1) - \eta(j+1) | \le \lambda ( | \zeta(j) - \eta(j) | + | g_j(\zeta) - g_j(\eta)| ) +
40 \lambda d^2 (C+1)^3 \frac{\log n}{\sqrt n}.
$$
Summing over $j=0, \dots, C-1$ now gives, provided $\lambda < \lambda_0(d) = 1/(8d+4)$,
\begin{eqnarray*}
\| \zeta-\eta \|_1 &\le& 2 \sum_{j=0}^{C-1} |\zeta(j+1) - \eta(j+1)| \\
&\le& 2 \lambda \big( \| \zeta - \eta \|_1 +  \| g(\zeta) - g(\eta) \|_1 \big) + 80 \lambda d^2 (C+1)^4 \frac{\log n}{\sqrt n} \\
&\le& (6d+2) \lambda \| \zeta-\eta \|_1  + 80 \lambda d^2 (C+1)^4 \frac{\log n}{\sqrt n} \\
&\le& \frac 34 \| \zeta-\eta \|_1  + 80 \lambda d^2 (C+1)^4 \frac{\log n}{\sqrt n},
\end{eqnarray*}
and so
$$
\| \zeta-\eta \|_1 \le 320 \lambda d^2 (C+1)^4 \frac{\log n}{\sqrt n}.
$$

If $\eta$ and $\zeta$ are two elements of $\Delta^{C+1}$ satisfying~(\ref{eq.fixed-point}), then we can apply exactly the same argument with the
term $40\lambda d^2 (C+1)^3 \frac{\log n}{\sqrt n}$, coming from the right-hand-side of (\ref{eq.bound}), replaced by zero, and we deduce that
$\| \zeta - \eta \| = 0$ and so $\zeta = \eta$; this amounts to an application of the contraction mapping theorem.  This establishes
part~(3) of Theorem~\ref{thm.low-high}, stating that there is a unique solution $\eta^*$ of (\ref{eq.fixed-point}), in this range of $\lambda$.

We now conclude that any $\zeta\in \Delta^{C+1}$ satisfying~(\ref{eq.bound}) lies within $\ell_1$ distance
$320 \lambda d^2 (C+1)^4 \frac{\log n}{\sqrt n}$ of $\eta^*$, and so in particular, if $\widehat{Z}_t$ is a copy of the process in equilibrium, then
$$
\left| \frac{1}{n-1} \E_\pi f_{v,j}(\widehat{Z}_t) - \eta^*(j) \right| \le 320 \lambda d^2 (C+1)^4 \frac{\log n}{\sqrt n}
$$
for each $j$, and this implies part~(4) of Theorem~\ref{thm.low-high} in this regime.

\medskip


For $\lambda \ge \lambda_1(d,C)$, we take the same general approach.
In this case, we deduce from (\ref{eq.lip}) that, for any $\eta$ and $\zeta$ in $\Delta^{C+1}$,
$$
\| g(\zeta) - g(\eta) \|_1 \le 16C d \| \zeta-\eta \|_1 \max(\zeta(\le\! C\!-\!1), \eta(\le\!C\!-\!1)).
$$
Now we assume that $\eta$ satisfies the fixed-point equation (\ref{eq.fixed-point}), and $\zeta$ satisfies the approximate version
(\ref{eq.bound}).  We have that, for $j = 0, \dots, C-1$,
$$
\eta(j) \le \frac{j+1}{\lambda} \eta(j+1) \le \frac{C}{\lambda} \eta(j+1),
$$
and so $\eta(C-i) \le (C/\lambda)^i \eta(C)$, for $i =1, \dots, C$.  Thus
$$
\eta(C) \ge \frac{1}{\sum_{i=0}^C (C/\lambda)^i} \ge 1 - \frac{C}{\lambda},
$$
and so $\eta(\le\!C\!-\!1) \le C/\lambda$.
Essentially the same is true for $\zeta$:
$$
\zeta(j) \le \frac{C}{\lambda} \zeta(j+1) + 40 \lambda d^2 (C+1)^3 \frac{\log n}{\sqrt n},
$$
which leads to $\zeta(C-i) \le (C/\lambda)^i \zeta(C) + 80 \lambda d^2 (C+1)^3 \frac{\log n}{\sqrt n}$ for
each $i = 0, \dots, C$, and hence
$$
\zeta(\le\!C\!-\!1) \le \frac{C}{\lambda} + 80 \lambda d^2 (C+1)^4 \frac{\log n}{\sqrt n} \le \frac{2C}{\lambda},
$$
for sufficiently large $n$.
Therefore we obtain that
$$
\| g(\zeta) - g(\eta) \|_1 \le \frac{32C^2 d}{\lambda} \| \zeta-\eta\|_1 \le \frac14 \| \zeta-\eta \|_1,
$$
whenever $\lambda \ge \lambda_1(d,C) = 8000C^2d \log(240C^2d)$, $\eta \in \Delta^{C+1}$ satisfies (\ref{eq.fixed-point}) and $\zeta \in \Delta^{C+1}$
satisfies~(\ref{eq.bound}).  Now we note that, for $j =0, \dots, C-1$,
$$
|\zeta(j)-\eta(j)| \le \frac{j+1}{\lambda} |\zeta(j+1) - \eta(j+1)| + |g_j(\zeta) - g_j(\eta)| + 40d^2 (C+1)^3 \frac{\log n}{\sqrt n}.
$$
Summing over $j = 0, \dots, C-1$ yields
\begin{eqnarray*}
\| \zeta-\eta\|_1 &\le& 2 \sum_{j=0}^{C-1} |\zeta(j) - \eta(j)| \\
&\le& \frac{2C}{\lambda} \|\zeta-\eta\|_1 + 2 \| g(\zeta) - g(\eta) \| + 80d^2 (C+1)^4 \frac{\log n}{\sqrt n} \\
&\le& \frac{1}{4} \|\zeta-\eta\|_1 + \frac{1}{4} \|\zeta-\eta\|_1 + 80d^2 (C+1)^4 \frac{\log n}{\sqrt n}.
\end{eqnarray*}
Hence, provided $\lambda \ge \lambda_1(d,C)$, $\eta$ satisfies (\ref{eq.fixed-point}) and $\zeta$ satisfies (\ref{eq.bound}),
$$
\| \zeta-\eta\|_1 \le 160d^2 (C+1)^4 \frac{\log n}{\sqrt n}.
$$
The same technique shows that there is only one solution $\eta^* \in \Delta^{C+1}$ to the fixed-point equation~(\ref{eq.fixed-point}).  Since the
vector $\zeta$ given by $\zeta(j) = \frac{1}{n-1} \E f_{v,j}(\widehat{Z}_t)$ ($j=0, \dots, C$), where $\widehat{Z}$ is in equilibrium, lies in $\Delta^{C+1}$
and satisfies (\ref{eq.bound}), we then obtain that
$$
\left| \frac{1}{n-1} \E_\pi f_{v,j}(\widehat{Z}_t) - \eta^*(j) \right| \le 160d^2 (C+1)^4 \frac{\log n}{\sqrt n}.
$$
This gives Theorem~\ref{thm.low-high}(3) and~(4) in this regime, completing the proof.

\end{document}